\pdfoutput=1
\documentclass[submission,copyright,creativecommons]{eptcs}
\providecommand{\event}{QPL 2011}
\def\authorrunning{A. Lang \& B. Coecke}
\def\titlerunning{Trichromatic Open Digraphs for Understanding Qubits}

\usepackage[latin1]{inputenc}
\usepackage[svgnames]{xcolor}
\usepackage{tikz}
\usepackage{stmaryrd}
\usepackage{amsmath}
\usepackage{amssymb}
\usepackage{amsthm}
\usepackage{braket}
\usepackage[all]{xy}
\usepackage{pigpen}
\usepackage{cite}

\usepackage{mparhack}

\usepgflibrary{shapes.geometric}

\DeclareMathOperator{\id}{id}
\newcommand{\cat}[1]{\mathbf{#1}}
\newcommand{\rgb}{\cat{RGB}}

\newcommand{\rg}{\cat{RG}}
\newcommand{\rgp}{\rg^+}
\newcommand{\fdhilb}{\cat{FdHilb}}
\newcommand{\stab}{\cat{Stab}}
\newcommand{\spek}{\cat{Spek}}

\newcommand{\dcirc}{\cat{D_\bigcirc}}
\newcommand{\dsmc}{SM$\dag$-category}

\newcommand{\sembrack}[1]{\left\llbracket #1 \right\rrbracket}

\newcommand{\tostab}[2]{
\mathchoice
{
  \sembrack{#2}_{\makebox[0pt][l]{${}_{#1}$}}
}
{
  \sembrack{#2}_{#1}
}{\sembrack{#2}_{#1}}{\sembrack{#2}_{#1}}
}

\newcommand{\rgtostab}[1]{\tostab{\rg}{#1}}
\newcommand{\rgbtostab}[1]{\tostab{\rgb}{#1}}
\newcommand{\rgptostab}[1]{\tostab{\rgp}{#1}}

\newcommand{\rgtorgb}{\mathcal{T}}
\newcommand{\rgtorgbof}[1]{\rgtorgb\left(#1\right)}
\newcommand{\rgbtorgp}{\mathcal{S}}
\newcommand{\rgbtorgpof}[1]{\rgbtorgp\left(#1\right)}
\newcommand{\dagof}[1]{\left(#1\right)^{\dag}}

\theoremstyle{plain}
\newtheorem{thm}{Theorem}[section]
\newtheorem{lem}[thm]{Lemma}
\newtheorem{prop}[thm]{Proposition}

\theoremstyle{definition}
\newtheorem{defn}{Definition}[section]

\theoremstyle{remark}

\newtheorem*{note}{Note}

\newcommand{\po}{\ar@{}[dr]|{\text{\pigpenfont R}}}
\newcommand{\pb}{\ar@{}[dr]|{\text{\pigpenfont J}}}

\tikzset{->, baseline=-2.5pt, >=latex}

\usetikzlibrary{decorations.markings}
\pgfdeclarelayer{edgelayer}
\pgfdeclarelayer{nodelayer}
\pgfsetlayers{edgelayer,nodelayer,main}

\tikzstyle{every picture}=[scale=0.75,inner sep=1pt, minimum size=11pt]
\tikzstyle{textstyle}=[scale=0.6,inner sep=1pt]
\newcommand{\textscale}{0.7}

\tikzstyle{none}=[inner sep=0pt,minimum size=0pt]
\tikzstyle{rn}=[ellipse,fill=White,draw=Red,line width=1 pt]
\tikzstyle{gn}=[ellipse,fill=White, draw=Lime, line width=1 pt]
\tikzstyle{bn}=[ellipse,fill=White, draw=Blue, line width=1 pt]

\tikzstyle{ro}=[circle,fill=Red!50, draw=black,line width=0 pt]
\tikzstyle{go}=[circle,fill=Lime!50, draw=black, line width=0 pt]

\tikzstyle{wn}=[circle,fill=White, draw=Black, line width=0 pt]
\tikzstyle{cr}=[decoration={markings, mark=at position \markat
with {\node [style=wn] {$\circlearrowright$};}}, postaction={decorate}]
\tikzstyle{cl}=[decoration={markings, mark=at position \markat
with {\node [style=wn] {$\circlearrowleft$};}}, postaction={decorate}]

\tikzstyle{ys}=[fill=Yellow, draw=Black, line width=0 pt]
\tikzstyle{hada}=[decoration={markings, mark=at position \markat
with {\node [style=ys] {$H$};}}, postaction={decorate}]

\newcommand{\markat}{0.5}

\newcommand{\markwithsym}{|}
\newcommand{\cyanmark}{{\arrow[cyan, line width=2pt]{\markwithsym}}}
\newcommand{\magentamark}{{\arrow[magenta, line width=2pt]{\markwithsym}}}
\newcommand{\yellowmark}{{\arrow[yellow, line width=2pt]{\markwithsym}}}
\newcommand{\testing}{{\draw[yellow,-,line width=2pt] (0pt, 5pt) -- (0pt, -5pt);}}
\tikzstyle{ctick}=[decoration={markings, mark=at position \markat
with \cyanmark},postaction={decorate}]
\tikzstyle{mtick}=[decoration={markings, mark=at position \markat
with \magentamark},postaction={decorate}]
\tikzstyle{ytick}=[decoration={markings, mark=at position \markat
with \yellowmark},postaction={decorate}]
\tikzstyle{ttick}=[decoration={markings, mark=at position \markat
with \testing},postaction={decorate}]

\newcommand{\identity}[1][]{
\mathchoice{\begin{tikzpicture}
	\begin{pgfonlayer}{nodelayer}
		\node [style=none] (0) at (0, 1) {};
		\node [style=none] (1) at (0, -1) {};
	\end{pgfonlayer}
	\begin{pgfonlayer}{edgelayer}
		\draw [#1] (0.center) to (1.center);
	\end{pgfonlayer}
\end{tikzpicture}}
{
\begin{tikzpicture}[scale=\textscale]
	\begin{pgfonlayer}{nodelayer}
		\node [style=none] (0) at (0, 1) {};
		\node [style=none] (1) at (0, -1) {};
	\end{pgfonlayer}
	\begin{pgfonlayer}{edgelayer}
		\draw [#1] (0.center) to (1.center);
	\end{pgfonlayer}
\end{tikzpicture}
}
{}{}
}

\newcommand{\hd}{
\mathchoice{\begin{tikzpicture}
	\begin{pgfonlayer}{nodelayer}
		\node [style=none] (0) at (0, 1) {};
		\node [style=none] (1) at (0, -1) {};
	\end{pgfonlayer}
	\begin{pgfonlayer}{edgelayer}
		\draw [hada] (0.center) to (1.center);
	\end{pgfonlayer}
\end{tikzpicture}}
{
\begin{tikzpicture}[scale=\textscale]
	\begin{pgfonlayer}{nodelayer}
		\node [style=none] (0) at (0, 1) {};
		\node [style=none] (1) at (0, -1) {};
	\end{pgfonlayer}
	\begin{pgfonlayer}{edgelayer}
		\draw [hada] (0.center) to (1.center);
	\end{pgfonlayer}
\end{tikzpicture}
}
{}{}
}

\newcommand{\rot}[2][$\theta$]{
  \mathchoice {
    \begin{tikzpicture}
      \begin{pgfonlayer}{nodelayer}
		\node [style=none] (0) at (0, 1) {};
		\node [style=#2] (1) at (0, -0) {#1};
		\node [style=none] (2) at (0, -1) {};
      \end{pgfonlayer}
      \begin{pgfonlayer}{edgelayer}
		\draw (0.center) to (1);
		\draw (1) to (2.center);
      \end{pgfonlayer}
    \end{tikzpicture}
  }
  {
    \begin{tikzpicture}[scale=\textscale]
      \begin{pgfonlayer}{nodelayer}
		\node [style=none] (0) at (0, 1) {};
		\node [style=#2] (1) at (0, -0) {#1};
		\node [style=none] (2) at (0, -1) {};
      \end{pgfonlayer}
      \begin{pgfonlayer}{edgelayer}
		\draw (0.center) to (1);
		\draw (1) to (2.center);
      \end{pgfonlayer}
    \end{tikzpicture}
  }
  {}{}
}

\newcommand\unit[2][]{
  \mathchoice
  {\begin{tikzpicture}
      \begin{pgfonlayer}{nodelayer}
		\node [style=#2] (0) at (0, 0.4) {#1};
		\node [style=none] (1) at (0, -0.6) {};
      \end{pgfonlayer}
      \begin{pgfonlayer}{edgelayer}
		\draw (0) to (1.center);
      \end{pgfonlayer}
    \end{tikzpicture}
  }
  {\begin{tikzpicture}[scale=\textscale]
	\begin{pgfonlayer}{nodelayer}
		\node [style=#2] (0) at (0, 0.4) {#1};
		\node [style=none] (1) at (0, -0.6) {};
	\end{pgfonlayer}
	\begin{pgfonlayer}{edgelayer}
		\draw (0) to (1.center);
	\end{pgfonlayer}
\end{tikzpicture}
  }
  {}
  {}
}

\newcommand\mul[2][]{
  \mathchoice
  {\begin{tikzpicture}
      \begin{pgfonlayer}{nodelayer}
		\node [style=none] (0) at (-0.8, 0.8) {};
		\node [style=none] (1) at (0.8, 0.8) {};
		\node [style=#2] (2) at (0, 0.2) {#1};
		\node [style=none] (3) at (0, -0.8) {};
      \end{pgfonlayer}
      \begin{pgfonlayer}{edgelayer}
		\draw (1.center) to (2);
		\draw (0.center) to (2);
		\draw (2) to (3.center);
      \end{pgfonlayer}
    \end{tikzpicture}
  }
  {
    \begin{tikzpicture}[scale=\textscale]
      \begin{pgfonlayer}{nodelayer}
        \node [style=none] (0) at (-0.8, 0.8) {};
        \node [style=none] (1) at (0.8, 0.8) {};
        \node [style=#2] (2) at (0, 0.2) {#1};
        \node [style=none] (3) at (0, -0.8) {};
      \end{pgfonlayer}
      \begin{pgfonlayer}{edgelayer}
        \draw (1.center) to (2);
        \draw (0.center) to (2);
        \draw (2) to (3.center);
      \end{pgfonlayer}
    \end{tikzpicture}
  }
  {}{}
}

\newcommand\counit[2][]{
  \mathchoice
  {
    \begin{tikzpicture}
      \begin{pgfonlayer}{nodelayer}
		\node [style=none] (0) at (0, 0.6) {};
		\node [style=#2] (1) at (0, -0.4) {#1};
      \end{pgfonlayer}
      \begin{pgfonlayer}{edgelayer}
		\draw (0.center) to (1);
      \end{pgfonlayer}
    \end{tikzpicture}
  }
  {\begin{tikzpicture}[scale=\textscale]
      \begin{pgfonlayer}{nodelayer}
		\node [style=none] (0) at (0, 0.6) {};
		\node [style=#2] (1) at (0, -0.4) {#1};
      \end{pgfonlayer}
      \begin{pgfonlayer}{edgelayer}
		\draw (0.center) to (1);
      \end{pgfonlayer}
    \end{tikzpicture}
  }
  {}{}
}

\newcommand\comul[2][]{
  \mathchoice
  {\begin{tikzpicture}
	\begin{pgfonlayer}{nodelayer}
		\node [style=none] (0) at (0, 0.8) {};
		\node [style=#2] (1) at (0, -0.2) {#1};
		\node [style=none] (2) at (-0.8, -0.8) {};
		\node [style=none] (3) at (0.8, -0.8) {};
	\end{pgfonlayer}
	\begin{pgfonlayer}{edgelayer}
		\draw (1) to (3.center);
		\draw (0.center) to (1);
		\draw (1) to (2.center);
	\end{pgfonlayer}
\end{tikzpicture}
}
  {\begin{tikzpicture}[scale=\textscale]
	\begin{pgfonlayer}{nodelayer}
		\node [style=none] (0) at (0, 0.8) {};
		\node [style=#2] (1) at (0, -0.2) {#1};
		\node [style=none] (2) at (-0.8, -0.8) {};
		\node [style=none] (3) at (0.8, -0.8) {};
	\end{pgfonlayer}
	\begin{pgfonlayer}{edgelayer}
		\draw (1) to (3.center);
		\draw (0.center) to (1);
		\draw (1) to (2.center);
	\end{pgfonlayer}
\end{tikzpicture}
}
  {}{}
}

\newcommand{\generators}[2]{
  \unit{#1}#2
  \mul{#1}#2
  \rot{#1}#2
  \comul{#1}#2
  \counit{#1}
}

\newcommand{\qqq}{\qquad\qquad}

\author{Alex Lang \qquad\qquad Bob Coecke
\institute{Department of Computer Science
\\ University of Oxford
\\ Oxford, UK}
\email{alexandre.lang@polytechnique.edu \qquad\qquad coecke@cs.ox.ac.uk}
}
\title{Trichromatic Open Digraphs for Understanding Qubits}

\begin{document}

\maketitle
\begin{abstract}
We introduce a trichromatic graphical calculus for quantum computing. The generators represent three complementary observables that are treated on equal footing, hence reflecting the symmetries of the Bloch sphere.  We derive the Euler angle decomposition of the Hadamard gate within it ~\cite{duncan2009graph} as well as the so-called supplementary relationships \cite{coecke2011three}, which are valid equations for qubits that were not derivable within $Z$/$X$-calculus  of Coecke and Duncan \cite{coecke2011interacting,coecke2008interacting}. More specifically, we have: dichromatic $Z$/$X$-calculus + Euler angle decomposition of the Hadamard gate = trichromatic calculus.
\end{abstract}


\section{Introduction}

We build on the stream of work of categorical semantics for quantum 
information processing, first initiated in ~\cite{abramsky2004categorical}. In this tradition, Coecke and Duncan 
developed and made extensive use of a calculus of dichromatic open digraphs to 
express quantum protocols and quantum 
states~\cite{coecke2008interacting,coecke2011interacting}. 
This graphical calculus turned out to be universal for quantum computing. 
The graphical calculus has been used to prove many statements useful to 
quantum computing, including some about measurement-based quantum computation (MBQC)~\cite{coecke2008interacting, duncan2009graph, duncan2010rewriting}, 
topological MBQC ~\cite{1367-2630-13-9-095011} and a 
multitude of other algorithms and protocols~\cite{coecke2011interacting,coecke2010environment,hillebrand2011quantum}. 

\begin{note}
  The term open digraphs is inspired by the work in ~\cite{DBLP:journals/corr/abs-1011-4114}, and 
  was chosen over the term diagram in order to avoid confusion with commutative 
  diagrams which are omnipresent in category theory, and will also be used in 
  this paper.
\end{note}

\begin{note}
  We will sometimes refer to the dichromatic calculus or the red-green 
  calculus, by which we mean the graphical calculus loosely defined in 
  ~\cite{coecke2008interacting,coecke2011interacting,
    duncan2009graph,duncan2010rewriting}
  ---sometimes called the $Z$/$X$-calculus---and which will be more precisely 
  defined as a category below.
\end{note}

The dichromatic calculus talks of two complementary observables in 
qubits. It is well known that it is possible to fit 3 complementary 
observables in qubits, and no more~\cite{DBLP:journals/qic/WocjanB05}. We thus set out to 
find a calculus which speaks of three complementary observables in a nice 
way---with the hope of developing a more ``complete'' theory---and contrast it 
with the existing dichromatic calculus, which is known to be incomplete with 
respect to stabilizer quantum mechanics. This incompleteness can be witnessed first-hand in
~\cite{coecke2011three,hillebrand2011quantum} where ad hoc rules were added to the dichromatic calculus.

The Bloch sphere is a well-understood~\cite{nielsen2002quantum} way to 
understand single-qubit unitaries as rotations of a 2-sphere---the elements of 
the group $SO(3)$. Although the $Z$/$X$-calculus is universal, and thus able
to express all these unitaries, it does not express some of the crucial
equations between them.

The symmetries of the Bloch sphere place each of the three complementary 
observables we wish to speak of on the same footing.  it would be desirable 
that this fact be reflected in the trichromatic calculus.  However, it can 
easily be shown that the Z/X-calculus does not allow for an extension which would make this true.  
More specifically, the generators of the $Z$/$X$-calculus have been picked 
such that the induced compact structures coincide, and one can easily show 
that this is only possible for a pair of complementary observables, and not a 
triple~\cite{coecke2008bases}.

We developed the trichromatic calculus, which expresses the presence of 3 
complementary observables in qubits, and also how they relate to each other 
through rotations of the Bloch sphere. We also showed that this was equivalent 
to adding an Euler angle decomposition of the Hadamard gate to the dichromatic 
calculus. This is in turn equivalent to Van den Nest's 
theorem in the dichromatic setting~\cite{van2004graphical,duncan2009graph}.
Thus, it subsumes all the improvements that are currently known to the 
dichromatic calculus, and does so in an elegant way. 
It is not yet known if these calculi are complete or not with 
respect to stabilizer quantum mechanics. It does seem, 
however, like the trichromatic calculus offers greater promise than the 
ad-hoc addition of the Euler angle decomposition to the dichromatic calculus.

Furthermore, with recent advances in Quantomatic\footnote{\url{https://sites.google.com/site/quantomatic/}}, it is now possible to automate rewriting in the trichromatic calculus.

\section{Preliminaries} 

We assume the reader is comfortable with the basic notions of category theory 
and quantum information theory, as well as Dirac notation and standard bases 
in $\mathbb C^2$. Here we present some additional definitions and notation 
which will be useful throughout.

\begin{defn}[\dsmc{}]
  A symmetric monoidal $\dag$-category is a symmetric monoidal 
  category $\cat C$ equipped with a contravariant involutive endofunctor 
  $\dagof{\cdot}:{\cat C} \to \cat C$, which acts as the identity on 
  objects and preserves the symmetric monoidal 
  structure~\cite{selinger2007dagger}.
\end{defn}

\begin{note}
  We will consider all our monoidal categories to be strict monoidal 
  categories. By Mac Lane's strictification theorem,
  any monoidal category is monoidally equivalent to a strict monoidal category.
\end{note}

\begin{defn}[$\fdhilb$]
  The \dsmc{} of finite-dimensional complex Hilbert spaces and linear maps between them.
\end{defn}

\begin{defn}[$\fdhilb_{wp}$]
  The \dsmc{} of finite-dimensional complex Hilbert spaces and linear maps modulo the relation 
  $f\equiv g$ if $\exists z\in \mathbb{C}, z\neq 0: f=zg$. We do this mainly to simplify the following exposition.
\end{defn}

\begin{defn}[$\fdhilb_Q$]
   The full subcategory of $\fdhilb_{wp}$ generated by the objects 
   \[\Set{\underbrace{Q\otimes \cdots \otimes Q}_n | n\geq 0},\] where 
   $Q:=\mathbb{C}^2$. This is essentially the category of qubits.
\end{defn}

We also use the following category which incarnates stabilizer quantum 
mechanics---as defined in~\cite{nielsen2002quantum}---and define it in a 
similar fashion as in ~\cite{coecke2011phase}.

\begin{defn}[$\stab$]
  The subcategory of $\fdhilb_Q$ generated by the following linear maps:
  \begin{itemize}
  \item
    single-qubit Clifford unitaries $:Q\to Q$
  \item
    $
      \delta_\stab:Q\to Q \otimes Q =
      \begin{cases}
        \ket{0}& \mapsto \ket{00}\\
        \ket{1}& \mapsto \ket{11}
      \end{cases}
    $
  \item
    $
      \varepsilon_\stab:Q \to 1=
      \begin{cases}
        \ket{+} &\mapsto 1\\
        \ket{-} &\mapsto 0
      \end{cases}
    $
  \end{itemize}
\end{defn}

\begin{defn}[$C_4$]
  The rotation group of the square. Of abstract group type $\mathbb{Z}/4\mathbb{Z}$ with elements denoted $\set{0,1,2,3}$ and its operation written additively.
\end{defn}

\section{Red and Green graphs}
\subsection{$\rg$ generators}
In a similar manner to the work of ~\cite{duncan2009graph,duncan2010rewriting}, we formalize the dichromatic calculus---described in detail in ~\cite{coecke2008interacting,coecke2011interacting}---as a symmetric monoidal category. We define a category $\rg$ where the objects are $n$-fold monoidal products of an object $*$, denoted $*^n$. In $\rg$, a morphism from $*^m$ to $*^n$ is a dichromatic open digraph from $m$ wires to $n$ wires, built from the generators below:
\begin{equation}
\generators{go}{\quad}
\qquad\quad
\hd
\qquad\quad
\generators{ro}{\quad}
\end{equation}
where $\theta$ was allowed to be any real number in ~\cite{coecke2011interacting}, we restrict $\theta$ to take values in $C_4$, restricting it to 4 values. Except for $\hd$, which is considered colourless, each generator is of one of two colours---hence the name dichromatic. Additionally, the identity morphism on $*$ is represented as the straight wire $\identity$. The generators $\rot{go}$ and $\rot{ro}$ are called phase gates. Composition is performed by plugging open ended wires together. One can notice that there are two types of open-ended wires. Ones going into a graph and ones going out of one. Informally, this corresponds to the ``input'' and ``output'' of the graph. We also mention here that we ignore connected components of a graph which are connected to neither input nor output. This is in order to not have to deal with scalars. 

\subsection{$\rg$ relations}

$\rg$ morphisms are also subject to the equations depicted below. The motivations behind these rules are explained in detail in ~\cite{coecke2008interacting,coecke2011interacting}.
\begin{equation}
  \text{Only the graph topology matters.}
\end{equation}
\begin{equation}
  \text{All equations hold under flip of arrows and negation of angles ($\dag$).}
  \label{eq:rgdag}
\end{equation}
\begin{equation}
\begin{tikzpicture}
	\begin{pgfonlayer}{nodelayer}
		\node [style=none] (0) at (-0.75, 1) {};
		\node [style=none] (1) at (0, 1) {$\cdots$};
		\node [style=none] (2) at (0.75, 1) {};
		\node [style=go] (3) at (0, -0) {$0$};
		\node [style=none] (4) at (-0.75, -1) {};
		\node [style=none] (5) at (0, -1) {$\cdots$};
		\node [style=none] (6) at (0.75, -1) {};
	\end{pgfonlayer}
	\begin{pgfonlayer}{edgelayer}
		\draw (3) to (6.center);
		\draw (0.center) to (3);
		\draw (2.center) to (3);
		\draw (3) to (4.center);
	\end{pgfonlayer}
\end{tikzpicture}
=
\begin{tikzpicture}
	\begin{pgfonlayer}{nodelayer}
		\node [style=none] (0) at (-0.75, 1) {};
		\node [style=none] (1) at (0, 1) {$\cdots$};
		\node [style=none] (2) at (0.75, 1) {};
		\node [style=go] (3) at (0, -0) {};
		\node [style=none] (4) at (-0.75, -1) {};
		\node [style=none] (5) at (0, -1) {$\cdots$};
		\node [style=none] (6) at (0.75, -1) {};
	\end{pgfonlayer}
	\begin{pgfonlayer}{edgelayer}
		\draw (3) to (6.center);
		\draw (0.center) to (3);
		\draw (2.center) to (3);
		\draw (3) to (4.center);
	\end{pgfonlayer}
\end{tikzpicture}
\qqq
  \begin{tikzpicture}
	\begin{pgfonlayer}{nodelayer}
		\node [style=none] (0) at (0, 1.25) {};
		\node [style=go] (1) at (0, -0) {};
		\node [style=none] (2) at (0, -1.25) {};
	\end{pgfonlayer}
	\begin{pgfonlayer}{edgelayer}
		\draw (0.center) to (1);
		\draw (1) to (2.center);
	\end{pgfonlayer}
\end{tikzpicture}
~=~
\identity
\qqq
  \begin{tikzpicture}
	\begin{pgfonlayer}{nodelayer}
		\node [style=none] (0) at (-2, 1) {};
		\node [style=none] (1) at (-1.25, 1) {$\cdots$};
		\node [style=none] (2) at (-0.5, 1) {};
		\node [style=none] (3) at (0.5, 1) {};
		\node [style=none] (4) at (1.25, 1) {$\cdots$};
		\node [style=none] (5) at (2, 1) {};
		\node [style=go] (6) at (-1.25, -0) {$\alpha$};
		\node [style=none] (7) at (0, -0) {$\vdots$};
		\node [style=go] (8) at (1.25, -0) {$\beta$};
		\node [style=none] (9) at (-2, -1) {};
		\node [style=none] (10) at (-1.25, -1) {$\cdots$};
		\node [style=none] (11) at (-0.5, -1) {};
		\node [style=none] (12) at (0.5, -1) {};
		\node [style=none] (13) at (1.25, -1) {$\cdots$};
		\node [style=none] (14) at (2, -1) {};
	\end{pgfonlayer}
	\begin{pgfonlayer}{edgelayer}
		\draw (6) to (11.center);
		\draw[bend left] (8) to (6);
		\draw[bend left] (6) to (8);
		\draw (8) to (12.center);
		\draw (6) to (9.center);
		\draw (5.center) to (8);
		\draw (3.center) to (8);
		\draw (2.center) to (6);
		\draw (0.center) to (6);
		\draw (8) to (14.center);
	\end{pgfonlayer}
\end{tikzpicture}
=
\begin{tikzpicture}
	\begin{pgfonlayer}{nodelayer}
		\node [style=none] (0) at (-1.5, 1) {};
		\node [style=none] (1) at (-1, 1) {$\cdots$};
		\node [style=none] (2) at (-0.5, 1) {};
		\node [style=none] (3) at (0.5, 1) {};
		\node [style=none] (4) at (1, 1) {$\cdots$};
		\node [style=none] (5) at (1.5, 1) {};
		\node [style=none] (6) at (1.5, 1) {};
		\node [style=go] (7) at (0, -0) {$\alpha + \beta$};
		\node [style=none] (8) at (-1.5, -1) {};
		\node [style=none] (9) at (-1, -1) {$\cdots$};
		\node [style=none] (10) at (-0.5, -1) {};
		\node [style=none] (11) at (0.5, -1) {};
		\node [style=none] (12) at (1, -1) {$\cdots$};
		\node [style=none] (13) at (1.5, -1) {};
	\end{pgfonlayer}
	\begin{pgfonlayer}{edgelayer}
		\draw (7) to (8.center);
		\draw (5.center) to (7);
		\draw (7) to (13.center);
		\draw (2.center) to (7);
		\draw (0.center) to (7);
		\draw (7) to (11.center);
		\draw (7) to (10.center);
		\draw (3.center) to (7);
	\end{pgfonlayer}
\end{tikzpicture}
\end{equation}

The three equations above imply that the quadruples of generators 
$\left(\unit{go},\mul{go},\comul{go},\counit{go}\right)$ and \\
$\left(\unit{ro},\mul{ro},\comul{ro},\counit{ro}\right)$ form $\dag$-special commutative Frobenius 
algebras as defined and exposed in~\cite{street:3930,kock2004frobenius,coecke2010compositional}.

\begin{gather}
\begin{tikzpicture}
	\begin{pgfonlayer}{nodelayer}
		\node [style=none] (0) at (-0.5, 1.5) {};
		\node [style=none] (1) at (0.5, 1.5) {};
		\node [style=go] (2) at (-0.5, 0.5) {};
		\node [style=go] (3) at (0.5, 0.5) {};
		\node [style=ro] (4) at (-0.5, -0.5) {};
		\node [style=ro] (5) at (0.5, -0.5) {};
		\node [style=none] (6) at (-0.5, -1.5) {};
		\node [style=none] (7) at (0.5, -1.5) {};
	\end{pgfonlayer}
	\begin{pgfonlayer}{edgelayer}
		\draw (2) to (5);
		\draw (2) to (4);
		\draw (4) to (6.center);
		\draw (3) to (5);
		\draw (0.center) to (2);
		\draw (3) to (4);
		\draw (5) to (7.center);
		\draw (1.center) to (3);
	\end{pgfonlayer}
\end{tikzpicture}
~=
\begin{tikzpicture}
	\begin{pgfonlayer}{nodelayer}
		\node [style=none] (0) at (-0.8, 1.1) {};
		\node [style=none] (1) at (0.8, 1.1) {};
		\node [style=ro] (2) at (0, 0.5) {};
		\node [style=go] (4) at (0, -0.5) {};
		\node [style=none] (5) at (-0.8, -1.1) {};
		\node [style=none] (6) at (0.8, -1.1) {};
	\end{pgfonlayer}
	\begin{pgfonlayer}{edgelayer}
		\draw (4) to (6.center);
		\draw (4) to (5.center);
		\draw (0.center) to (2);
		\draw (1.center) to (2);
		\draw (2) to (4);
	\end{pgfonlayer}
\end{tikzpicture}
\qqq
\begin{tikzpicture}
	\begin{pgfonlayer}{nodelayer}
		\node [style=ro] (0) at (0, 0.8) {};
		\node [style=go] (1) at (0, -0.2) {};
		\node [style=none] (2) at (-0.8, -0.8) {};
		\node [style=none] (3) at (0.8, -0.8) {};
	\end{pgfonlayer}
	\begin{pgfonlayer}{edgelayer}
		\draw (1) to (3.center);
		\draw (0.center) to (1);
		\draw (1) to (2.center);
	\end{pgfonlayer}
\end{tikzpicture}
=~
  \begin{tikzpicture}
	\begin{pgfonlayer}{nodelayer}
		\node [style=ro] (2) at (-0.5, 0.5) {};
		\node [style=ro] (3) at (0.5, 0.5) {};
		\node [style=none] (4) at (-0.5, -0.5) {};
		\node [style=none] (5) at (0.5, -0.5) {};
	\end{pgfonlayer}
	\begin{pgfonlayer}{edgelayer}
		\draw (2) to (4);
		\draw (3) to (5);
	\end{pgfonlayer}
\end{tikzpicture}
\qqq
\begin{tikzpicture}
	\begin{pgfonlayer}{nodelayer}
		\node [style=go] (0) at (0, 0.4) {};
		\node [style=none] (1) at (-0.6, -0.4) {};
		\node [style=none] (2) at (0.6, -0.4) {};
	\end{pgfonlayer}
	\begin{pgfonlayer}{edgelayer}
		\draw (0) to (2.center);
		\draw (0) to (1.center);
	\end{pgfonlayer}
\end{tikzpicture}
~=~
\begin{tikzpicture}
	\begin{pgfonlayer}{nodelayer}
		\node [style=ro] (0) at (0, 0.4) {};
		\node [style=none] (1) at (-0.6, -0.4) {};
		\node [style=none] (2) at (0.6, -0.4) {};
	\end{pgfonlayer}
	\begin{pgfonlayer}{edgelayer}
		\draw (0) to (2.center);
		\draw (0) to (1.center);
	\end{pgfonlayer}
\end{tikzpicture}
\\
  \begin{tikzpicture}
	\begin{pgfonlayer}{nodelayer}
		\node [style=none] (0) at (0, 1.3) {};
		\node [style=ro] (1) at (0, 0.3) {$2$};
		\node [style=go] (2) at (0, -0.7) {};
		\node [style=none] (3) at (-0.8, -1.3) {};
		\node [style=none] (4) at (0.8, -1.3) {};
	\end{pgfonlayer}
	\begin{pgfonlayer}{edgelayer}
		\draw (1) to (2);
		\draw (2) to (4.center);
		\draw (0.center) to (1);
		\draw (2) to (3.center);
	\end{pgfonlayer}
\end{tikzpicture}
=~
\begin{tikzpicture}
	\begin{pgfonlayer}{nodelayer}
		\node [style=none] (0) at (0, 1.3) {};
		\node [style=go] (1) at (0, 0.3) {};
		\node [style=ro] (2) at (-0.8, -0.3) {$2$};
		\node [style=ro] (3) at (0.8, -0.3) {$2$};
		\node [style=none] (4) at (-0.8, -1.3) {};
		\node [style=none] (5) at (0.8, -1.3) {};
	\end{pgfonlayer}
	\begin{pgfonlayer}{edgelayer}
		\draw (1) to (3);
		\draw (0.center) to (1);
		\draw (3) to (5.center);
		\draw (2) to (4.center);
		\draw (1) to (2);
	\end{pgfonlayer}
\end{tikzpicture}
\qqq
  \begin{tikzpicture}
	\begin{pgfonlayer}{nodelayer}
		\node [style=none] (0) at (0, 1.5) {};
		\node [style=go] (1) at (0, 0.5) {$2$};
		\node [style=ro] (2) at (0, -0.5) {$\theta$};
		\node [style=none] (3) at (0, -1.5) {};
	\end{pgfonlayer}
	\begin{pgfonlayer}{edgelayer}
		\draw (1) to (2);
		\draw (0.center) to (1);
		\draw (2) to (3.center);
	\end{pgfonlayer}
\end{tikzpicture}
~=~
\begin{tikzpicture}
	\begin{pgfonlayer}{nodelayer}
		\node [style=none] (0) at (0, 1.5) {};
		\node [style=ro] (1) at (0, 0.5) {$-\theta$};
		\node [style=go] (2) at (0, -0.5) {$2$};
		\node [style=none] (3) at (0, -1.5) {};
	\end{pgfonlayer}
	\begin{pgfonlayer}{edgelayer}
		\draw (1) to (2);
		\draw (0.center) to (1);
		\draw (2) to (3.center);
	\end{pgfonlayer}
\end{tikzpicture}
\qqq
\begin{tikzpicture}
	\begin{pgfonlayer}{nodelayer}
		\node [style=none] (0) at (0, 1) {};
		\node [style=none] (2) at (0, 0) {};
		\node [style=none] (3) at (0, -1) {};
	\end{pgfonlayer}
	\begin{pgfonlayer}{edgelayer}
		\draw [hada,-] (0.center) to (2.center);
		\draw [hada] (2.center) to (3.center);
	\end{pgfonlayer}
\end{tikzpicture}
~=~\identity
\qqq
  \begin{tikzpicture}
	\begin{pgfonlayer}{nodelayer}
		\node [style=none] (0) at (-0.6, 0.8) {};
		\node [style=none] (1) at (0, 0.8) {$\cdots$};
		\node [style=none] (2) at (0.6, 0.8) {};
		\node [style=go] (3) at (0, -0) {$\theta$};
		\node [style=none] (4) at (-0.6, -0.8) {};
		\node [style=none] (5) at (0, -0.8) {$\cdots$};
		\node [style=none] (6) at (0.6, -0.8) {};
	\end{pgfonlayer}
	\begin{pgfonlayer}{edgelayer}
		\draw (0.center) to (3);
		\draw (3) to (6.center);
		\draw (2.center) to (3);
		\draw (3) to (4.center);
	\end{pgfonlayer}
\end{tikzpicture}
~=
  \begin{tikzpicture}
	\begin{pgfonlayer}{nodelayer}
		\node [style=none] (0) at (-1.2, 1.6) {};
		\node [style=none] (1) at (0, 1.6) {$\cdots$};
		\node [style=none] (2) at (1.2, 1.6) {};
		\node [style=ro] (3) at (0, -0) {$\theta$};
		\node [style=none] (4) at (-1.2, -1.6) {};
		\node [style=none] (5) at (0, -1.6) {$\cdots$};
		\node [style=none] (6) at (1.2, -1.6) {};
	\end{pgfonlayer}
	\begin{pgfonlayer}{edgelayer}
		\draw [hada] (0.center) to (3);
		\draw [hada] (3) to (6.center);
		\draw [hada] (2.center) to (3);
		\draw [hada] (3) to (4.center);
	\end{pgfonlayer}
\end{tikzpicture}
\end{gather}

The spider rule gives the subcategory of $\rg$ generated by the phase gates a group structure. In this case, the group is $C_4*C_4$, the free product of 2 copies of $C_4$---one for each colour.

The last two rules can be used to prove the following equation.
\begin{equation}
  \text{All rules hold under flip of colours.}
\label{eq:rgcyc}
\end{equation}

\subsection{$\dag$ Structure}
The symmetric monoidal category $\rg$ can further be made into a \dsmc{} by having $\dag$ act on the generators like so:
\[\begin{array}{c@{\qquad}c@{\qquad}c@{\qquad}c@{\qquad}c}
  \dagof{\unit{go}}=~\counit{go}
&
  \dagof{\rot{go}}=~\rot[$-\theta$]{go}
&
  \dagof{\counit{go}}=~\unit{go}
&
  \dagof{\mul{go}}=\comul{go}
&
  \dagof{\comul{go}}=\mul{go}
\\
  \dagof{\unit{ro}}=~\counit{ro}
&
  \dagof{\rot{ro}}=~\rot[$-\theta$]{ro}
&
  \dagof{\counit{ro}}=~\unit{ro}
&
  \dagof{\mul{ro}}=\comul{ro}
&
  \dagof{\comul{ro}}=\mul{ro}
\end{array}
\]
where functoriality of $\dagof{\cdot}$ is guaranteed by Rule~\eqref{eq:rgdag}

\subsection{$\rg$ interpretation}
So far we have described a graphical category from generators and relations but with no explicit quantum content, but in fact these digraphs are meant to represent quantum maps. Here, we provide an interpretation for these digraphs by describing a monoidal functor $\rgtostab{\cdot}:\rg\to\stab$, taking $*$ to $\mathbb{C}^2$ and mapping the morphisms like so (as expressed in Dirac notation):

\begin{gather*}
    \rgtostab{\unit{go}} ~=~ \ket{+}
    \qqq
    \rgtostab{\rot{go}} ~=~ \ket0\bra0 ~+~ e^{i\frac{\pi}{2}\theta}\ket1\bra1
    \qqq
    \rgtostab{\counit{go}} ~=~ \bra{+}
  \\
  \rgtostab{\mul{go}} ~=~ \ket{0}\bra{00}~+~\ket{1}\bra{11}
  \qqq
  \rgtostab{\comul{go}} ~=~ \ket{00}\bra{0} ~+~\ket{11}\bra{1}
\\
    \rgtostab{\unit{ro}} ~=~ \ket{0}
  \qqq
  \rgtostab{\rot{ro}} ~=~ \ket+\bra+ ~+~ e^{i\frac{\pi}{2}\theta}\ket-\bra-
  \qqq
    \rgtostab{\counit{ro}} ~=~ \bra{0}
  \\
    \rgtostab{\comul{ro}} ~=~ \ket{++}\bra{+} ~+~ \ket{--}\bra{-}
    \qqq
    \rgtostab{\mul{ro}} ~=~ \ket{+}\bra{++} ~+~ \ket{-}\bra{--}
\\
\rgtostab{\hd}
  ~=~ \ket{+}\bra{0} ~+~ \ket{-}\bra{1}
\end{gather*}

The calculus of $\rg$ turns out to be universal, as exposed in~\cite{coecke2011interacting}.

\begin{prop}
  $\rgtostab{\cdot}$ is indeed a symmetric monoidal $\dag$-functor.
  \label{prop:rginter}
\end{prop}
 
\begin{proof}
  This involves checking for each rule $f=g$ in $\rg$ that 
  $\rgtostab{f}=\rgtostab{g}$, that $\rgtostab{\cdot}$ respects the symmetric
  monoidal structure on the generators, and for each generator $g$, we have 
  $\rgtostab{g}^\dag=\rgtostab{g^\dag}$.
\end{proof}

\begin{figure*}
  \centering
  \includegraphics[scale=0.32]{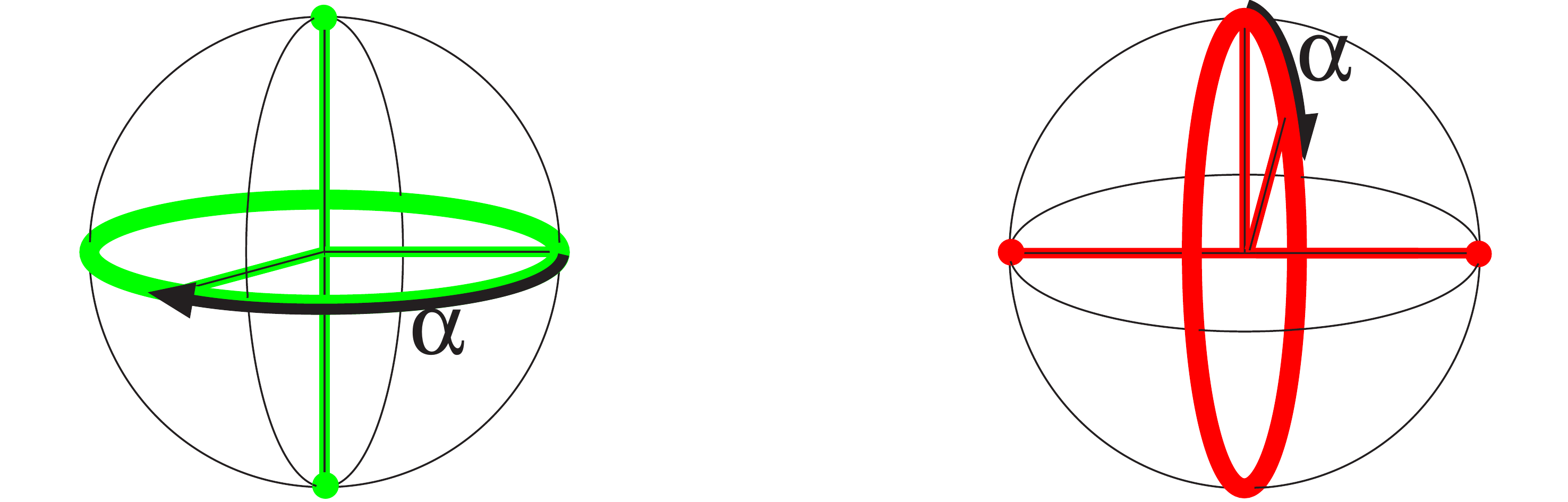}
  \caption{$\rg$ Bloch spheres}
  \label{fig:rgbloch}
\end{figure*}

Figure~\ref{fig:rgbloch} depicts the interpretation of the rotation gates $\rot{go}$ and $\rot{ro}$. The arrow-tail represents the locations of the deleting points $\unit{go}$ and $\unit{ro}$. Although the rotations in the axis which is perpendicular to both of these can be expressed in terms of red and green rotations, there is no primitive support for them. This is what we will attempt to remedy in the next section.

\section{Red, Green and Blue digraphs}
Here we introduce a theory of trichromatic diagraphs as a category $\rgb$. Whereas the digraphs in $\rg$ spoke of two complementary observable structures, $\rgb$ speaks of three complementary observable structures, the maximum number that can hold in qubits.

One might ask---why look for another diagrammatic theory if $\rg$ is already 
universal? The reason is expressed in the following proposition.

\begin{prop}
  $\rgtostab{\cdot}$ is unfaithful.
  \label{prop:unfaith}
\end{prop}

\begin{proof}
  A counterexample is provided by the following morphisms which are not equal in $\rg$, but whose images are equal in $\stab$.
\begin{equation}
\begin{tikzpicture}
	\begin{pgfonlayer}{nodelayer}
		\node [style=none] (0) at (0, 1.6) {};
		\node [style=go] (1) at (0, 0.6) {};
		\node [style=ro] (2) at (-0.8, -0) {$1$};
		\node [style=ro] (3) at (0.8, -0) {$1$};
		\node [style=go] (4) at (0, -0.6) {};
		\node [style=none] (5) at (0, -1.6) {};
	\end{pgfonlayer}
	\begin{pgfonlayer}{edgelayer}
		\draw (2) to (4);
		\draw (1) to (2);
		\draw (1) to (3);
		\draw (3) to (4);
		\draw (4) to (5.center);
		\draw (0.center) to (1);
	\end{pgfonlayer}
\end{tikzpicture}
~\neq~
\begin{tikzpicture}
	\begin{pgfonlayer}{nodelayer}
		\node [style=none] (0) at (0, 1.5) {};
		\node [style=go] (1) at (0, 0.5) {$2$};
		\node [style=go] (2) at (0, -0.5) {$2$};
		\node [style=none] (3) at (0, -1.5) {};
	\end{pgfonlayer}
	\begin{pgfonlayer}{edgelayer}
		\draw (0.center) to (1);
		\draw (2) to (3.center);
	\end{pgfonlayer}
\end{tikzpicture}
\qquad\text{ but }\qquad
\rgtostab{\begin{tikzpicture}
	\begin{pgfonlayer}{nodelayer}
		\node [style=none] (0) at (0, 1.6) {};
		\node [style=go] (1) at (0, 0.6) {};
		\node [style=ro] (2) at (-0.8, -0) {$1$};
		\node [style=ro] (3) at (0.8, -0) {$1$};
		\node [style=go] (4) at (0, -0.6) {};
		\node [style=none] (5) at (0, -1.6) {};
	\end{pgfonlayer}
	\begin{pgfonlayer}{edgelayer}
		\draw (2) to (4);
		\draw (1) to (2);
		\draw (1) to (3);
		\draw (3) to (4);
		\draw (4) to (5.center);
		\draw (0.center) to (1);
	\end{pgfonlayer}
\end{tikzpicture}}
~=~
\rgtostab{\begin{tikzpicture}
	\begin{pgfonlayer}{nodelayer}
		\node [style=none] (0) at (0, 1.5) {};
		\node [style=go] (1) at (0, 0.5) {$2$};
		\node [style=go] (2) at (0, -0.5) {$2$};
		\node [style=none] (3) at (0, -1.5) {};
	\end{pgfonlayer}
	\begin{pgfonlayer}{edgelayer}
		\draw (0.center) to (1);
		\draw (2) to (3.center);
	\end{pgfonlayer}
\end{tikzpicture}}
~=~
\ket{-}\bra{-}\qedhere
\end{equation}

\end{proof}

This is a special case of the supplementary rule which was used extensively in ~\cite{coecke2011three,hillebrand2011quantum}.

Effectively, this means that $\rg$ is not complete with respect to $\stab$. That is, there exist morphisms in $\rg$ which are not equal, but their interpretations are. Despite this, one might still ask why develop a new calculus rather than extend the existing one. The reason was alluded to in the introduction. It is mainly because the three complementary observables lie on an equal footing with respect to the Bloch sphere. However, if we look at the $Z$ and $X$ observables in the dichromatic calculus, they share the same compact structure. If we were to add a third $Y$ observable, it would not be able to share the same compact structure~\cite{coecke2008bases}, and would thus be the odd one out.

\subsection{The category $\rgb$ of trichromatic digraphs}

We define a category $\rgb$ of trichromatic digraphs in a similar manner as we did for the dichromatic category $\rg$. $\rgb$ is a \dsmc{} with as objects, tensor products of $*$. Morphisms in $\rgb$ are represented as open digraphs with coloured nodes and again, the identity morphism on $*$ is represented graphically by: $\identity$

The monoidal product is represented by horizontal disjoint union of digraphs and composition is represented by vertical plugging of wires.

$\rgb$ is generated by (horizontal) disjoint union and (vertical) composition by the following generators:
\begin{equation*}
\generators{rn}{~~}\qquad
\generators{gn}{~~}\qquad
\generators{bn}{~~}
\end{equation*}
where again $\theta$ can be any element of $C_4$. Of special note is the fact that there is no Hadamard gate in $\rgb$. There is instead a pair of gates which take the role of ``colour changers'' which can be defined in terms of the above generators. They will be defined in the next section.

\subsection{RGB Rules}

These following equations hold in $\rgb$.
\begin{equation}
  \text{Only the graph topology matters.}
\end{equation}
\begin{equation}
  \text{All equations hold under flip of arrows and negation of angles ($\dag$).}
  \label{eq:dag}
\end{equation}
\begin{equation}
\begin{tikzpicture}
	\begin{pgfonlayer}{nodelayer}
		\node [style=none] (0) at (-0.75, 1) {};
		\node [style=none] (1) at (0, 1) {$\cdots$};
		\node [style=none] (2) at (0.75, 1) {};
		\node [style=gn] (3) at (0, -0) {$0$};
		\node [style=none] (4) at (-0.75, -1) {};
		\node [style=none] (5) at (0, -1) {$\cdots$};
		\node [style=none] (6) at (0.75, -1) {};
	\end{pgfonlayer}
	\begin{pgfonlayer}{edgelayer}
		\draw (3) to (6.center);
		\draw (0.center) to (3);
		\draw (2.center) to (3);
		\draw (3) to (4.center);
	\end{pgfonlayer}
\end{tikzpicture}
~=~
\begin{tikzpicture}
	\begin{pgfonlayer}{nodelayer}
		\node [style=none] (0) at (-0.75, 1) {};
		\node [style=none] (1) at (0, 1) {$\cdots$};
		\node [style=none] (2) at (0.75, 1) {};
		\node [style=gn] (3) at (0, -0) {};
		\node [style=none] (4) at (-0.75, -1) {};
		\node [style=none] (5) at (0, -1) {$\cdots$};
		\node [style=none] (6) at (0.75, -1) {};
	\end{pgfonlayer}
	\begin{pgfonlayer}{edgelayer}
		\draw (3) to (6.center);
		\draw (0.center) to (3);
		\draw (2.center) to (3);
		\draw (3) to (4.center);
	\end{pgfonlayer}
\end{tikzpicture}
\qqq
  \begin{tikzpicture}
	\begin{pgfonlayer}{nodelayer}
		\node [style=none] (0) at (0, 1.25) {};
		\node [style=gn] (1) at (0, -0) {};
		\node [style=none] (2) at (0, -1.25) {};
	\end{pgfonlayer}
	\begin{pgfonlayer}{edgelayer}
		\draw (0.center) to (1);
		\draw (1) to (2.center);
	\end{pgfonlayer}
\end{tikzpicture}
~=~
\identity
\qqq
\begin{tikzpicture}
	\begin{pgfonlayer}{nodelayer}
		\node [style=none] (0) at (-2, 1) {};
		\node [style=none] (1) at (-1.25, 1) {$\cdots$};
		\node [style=none] (2) at (-0.5, 1) {};
		\node [style=none] (3) at (0.5, 1) {};
		\node [style=none] (4) at (1.25, 1) {$\cdots$};
		\node [style=none] (5) at (2, 1) {};
		\node [style=gn] (6) at (-1.25, -0) {$\alpha$};
		\node [style=none] (7) at (0, -0) {$\vdots$};
		\node [style=gn] (8) at (1.25, -0) {$\beta$};
		\node [style=none] (9) at (-2, -1) {};
		\node [style=none] (10) at (-1.25, -1) {$\cdots$};
		\node [style=none] (11) at (-0.5, -1) {};
		\node [style=none] (12) at (0.5, -1) {};
		\node [style=none] (13) at (1.25, -1) {$\cdots$};
		\node [style=none] (14) at (2, -1) {};
	\end{pgfonlayer}
	\begin{pgfonlayer}{edgelayer}
		\draw (6) to (11.center);
		\draw[bend left] (8) to (6);
		\draw[bend left] (6) to (8);
		\draw (8) to (12.center);
		\draw (6) to (9.center);
		\draw (5.center) to (8);
		\draw (3.center) to (8);
		\draw (2.center) to (6);
		\draw (0.center) to (6);
		\draw (8) to (14.center);
	\end{pgfonlayer}
\end{tikzpicture}
~=~
\begin{tikzpicture}
	\begin{pgfonlayer}{nodelayer}
		\node [style=none] (0) at (-1.5, 1) {};
		\node [style=none] (1) at (-1, 1) {$\cdots$};
		\node [style=none] (2) at (-0.5, 1) {};
		\node [style=none] (3) at (0.5, 1) {};
		\node [style=none] (4) at (1, 1) {$\cdots$};
		\node [style=none] (5) at (1.5, 1) {};
		\node [style=none] (6) at (1.5, 1) {};
		\node [style=gn] (7) at (0, -0) {$\alpha + \beta$};
		\node [style=none] (8) at (-1.5, -1) {};
		\node [style=none] (9) at (-1, -1) {$\cdots$};
		\node [style=none] (10) at (-0.5, -1) {};
		\node [style=none] (11) at (0.5, -1) {};
		\node [style=none] (12) at (1, -1) {$\cdots$};
		\node [style=none] (13) at (1.5, -1) {};
	\end{pgfonlayer}
	\begin{pgfonlayer}{edgelayer}
		\draw (7) to (8.center);
		\draw (5.center) to (7);
		\draw (7) to (13.center);
		\draw (2.center) to (7);
		\draw (0.center) to (7);
		\draw (7) to (11.center);
		\draw (7) to (10.center);
		\draw (3.center) to (7);
	\end{pgfonlayer}
\end{tikzpicture}
\end{equation}

Again, the laws above allow us to state that each of the quadruples of generators $\left(\unit{rn},\mul{rn},\comul{rn},\counit{rn}\right)$, $\left(\unit{gn},\mul{gn},\comul{gn},\counit{gn}\right)$ and $\left(\unit{bn},\mul{bn},\comul{bn},\counit{bn}\right)$ form Frobenius algebras.
\begin{equation}
\begin{tikzpicture}
	\begin{pgfonlayer}{nodelayer}
		\node [style=bn] (0) at (0, 0.8) {};
		\node [style=gn] (1) at (0, -0.2) {};
		\node [style=none] (2) at (-0.8, -0.8) {};
		\node [style=none] (3) at (0.8, -0.8) {};
	\end{pgfonlayer}
	\begin{pgfonlayer}{edgelayer}
		\draw (1) to (3.center);
		\draw (0.center) to (1);
		\draw (1) to (2.center);
	\end{pgfonlayer}
\end{tikzpicture}
~=~
  \begin{tikzpicture}
	\begin{pgfonlayer}{nodelayer}
		\node [style=bn] (2) at (-0.5, 0.5) {};
		\node [style=bn] (3) at (0.5, 0.5) {};
		\node [style=none] (4) at (-0.5, -0.5) {};
		\node [style=none] (5) at (0.5, -0.5) {};
	\end{pgfonlayer}
	\begin{pgfonlayer}{edgelayer}
		\draw (2) to (4);
		\draw (3) to (5);
	\end{pgfonlayer}
\end{tikzpicture}
\qqq
\begin{tikzpicture}
      \begin{pgfonlayer}{nodelayer}
		\node [style=none] (0) at (-0.8, 0.8) {};
		\node [style=none] (1) at (0.8, 0.8) {};
		\node [style=rn] (2) at (0, 0.2) {$3$};
		\node [style=gn] (3) at (0, -0.8) {};
      \end{pgfonlayer}
      \begin{pgfonlayer}{edgelayer}
		\draw (1.center) to (2);
		\draw (0.center) to (2);
		\draw (2) to (3);
      \end{pgfonlayer}
    \end{tikzpicture}
~=~
  \begin{tikzpicture}
	\begin{pgfonlayer}{nodelayer}
		\node [style=none] (2) at (-0.5, 0.5) {};
		\node [style=none] (3) at (0.5, 0.5) {};
		\node [style=gn] (4) at (-0.5, -0.5) {};
		\node [style=gn] (5) at (0.5, -0.5) {};
	\end{pgfonlayer}
	\begin{pgfonlayer}{edgelayer}
		\draw (2) to (4);
		\draw (3) to (5);
	\end{pgfonlayer}
  \end{tikzpicture}
\qqq
\begin{tikzpicture}
	\begin{pgfonlayer}{nodelayer}
		\node [style=none] (0) at (-0.5, 1.5) {};
		\node [style=none] (1) at (0.5, 1.5) {};
		\node [style=gn] (2) at (-0.5, 0.5) {};
		\node [style=gn] (3) at (0.5, 0.5) {};
		\node [style=rn] (4) at (-0.5, -0.5) {$3$};
		\node [style=rn] (5) at (0.5, -0.5) {$3$};
		\node [style=none] (6) at (-0.5, -1.5) {};
		\node [style=none] (7) at (0.5, -1.5) {};
	\end{pgfonlayer}
	\begin{pgfonlayer}{edgelayer}
		\draw (2) to (5);
		\draw (2) to (4);
		\draw (4) to (6.center);
		\draw (3) to (5);
		\draw (0.center) to (2);
		\draw (3) to (4);
		\draw (5) to (7.center);
		\draw (1.center) to (3);
	\end{pgfonlayer}
\end{tikzpicture}
~=
\begin{tikzpicture}
	\begin{pgfonlayer}{nodelayer}
		\node [style=none] (0) at (-0.8, 1.1) {};
		\node [style=none] (1) at (0.8, 1.1) {};
		\node [style=rn] (2) at (0, 0.5) {$3$};
		\node [style=gn] (4) at (0, -0.5) {};
		\node [style=none] (5) at (-0.8, -1.1) {};
		\node [style=none] (6) at (0.8, -1.1) {};
	\end{pgfonlayer}
	\begin{pgfonlayer}{edgelayer}
		\draw (4) to (6.center);
		\draw (4) to (5.center);
		\draw (0.center) to (2);
		\draw (1.center) to (2);
		\draw (2) to (4);
	\end{pgfonlayer}
\end{tikzpicture}
\end{equation}

These laws state that state that the quadruple 
$\left(\unit{bn}, \mul[$3$]{rn}, \comul{gn}, \counit{gn}\right)$ forms a 
bialgebra, as defined in~\cite{kassel1995quantum}. Using Rule~\eqref{eq:cyc} and Rule~\eqref{eq:dag}, we can also show 
that the following quadruples form bialgebras: 
$\left(\unit{gn}, \mul[$3$]{bn}, \comul{rn}, \counit{rn}\right)$, 
$\left(\unit{rn}, \mul[$3$]{gn}, \comul{bn}, \counit{bn}\right)$, 
$\left(\unit{gn}, \mul{gn}, \comul[$1$]{rn}, \counit{bn}\right)$, 
$\left(\unit{rn}, \mul{rn}, \comul[$1$]{bn}, \counit{gn}\right)$, 
$\left(\unit{bn}, \mul{bn}, \comul[$1$]{gn}, \counit{rn}\right)$. 
These laws are perhaps not as nice as the bialgebra laws from $\rg$, but this 
is the price to pay for a theory that reflects the symmetries of the Bloch sphere.
\begin{equation}
\begin{tikzpicture}
	\begin{pgfonlayer}{nodelayer}
		\node [style=none] (0) at (0, 1) {};
		\node [style=gn] (1) at (0, -0) {};
		\node [style=rn] (2) at (1, -0) {};
		\node [style=none] (3) at (0, -1) {};
	\end{pgfonlayer}
	\begin{pgfonlayer}{edgelayer}
		\draw (2) to (1);
		\draw (1) to (3.center);
		\draw (0.center) to (1);
	\end{pgfonlayer}
\end{tikzpicture}
~=~
  \begin{tikzpicture}
	\begin{pgfonlayer}{nodelayer}
		\node [style=none] (0) at (0, 1) {};
		\node [style=gn] (1) at (0, -0) {$1$};
		\node [style=none] (2) at (0, -1) {};
	\end{pgfonlayer}
	\begin{pgfonlayer}{edgelayer}
		\draw (1) to (2.center);
		\draw (0.center) to (1);
	\end{pgfonlayer}
\end{tikzpicture}
\end{equation}

The above equation allows us to express a rotation in term of a multiplication or comultiplication.

\begin{gather}
\begin{tikzpicture}
	\begin{pgfonlayer}{nodelayer}
		\node [style=none] (0) at (0, 1) {};
		\node [style=none] (1) at (0, -1) {};
	\end{pgfonlayer}
	\begin{pgfonlayer}{edgelayer}
		\draw [cr] (0.center) to (1.center);
	\end{pgfonlayer}
\end{tikzpicture}
~:=~
  \begin{tikzpicture}
	\begin{pgfonlayer}{nodelayer}
		\node [style=none] (0) at (0, 1.5) {};
		\node [style=gn] (1) at (0, 0.5) {$1$};
		\node [style=bn] (3) at (0, -0.5) {$1$};
		\node [style=none] (5) at (0, -1.5) {};
	\end{pgfonlayer}
	\begin{pgfonlayer}{edgelayer}
		\draw (0.center) to (1);
		\draw (3) to (5.center);
		\draw (1) to (3);
	\end{pgfonlayer}
\end{tikzpicture}
~=~
\begin{tikzpicture}
	\begin{pgfonlayer}{nodelayer}
		\node [style=none] (0) at (0, 1.5) {};
		\node [style=rn] (1) at (0, 0.5) {$1$};
		\node [style=gn] (3) at (0, -0.5) {$1$};
		\node [style=none] (5) at (0, -1.5) {};
	\end{pgfonlayer}
	\begin{pgfonlayer}{edgelayer}
		\draw (0.center) to (1);
		\draw (3) to (5.center);
		\draw (1) to (3);
	\end{pgfonlayer}
\end{tikzpicture}
~=~
\begin{tikzpicture}
	\begin{pgfonlayer}{nodelayer}
		\node [style=none] (0) at (0, 1.5) {};
		\node [style=bn] (1) at (0, 0.5) {$1$};
		\node [style=rn] (3) at (0, -0.5) {$1$};
		\node [style=none] (5) at (0, -1.5) {};
	\end{pgfonlayer}
	\begin{pgfonlayer}{edgelayer}
		\draw (0.center) to (1);
		\draw (3) to (5.center);
		\draw (1) to (3);
	\end{pgfonlayer}
\end{tikzpicture}
\qqq
\begin{tikzpicture}
	\begin{pgfonlayer}{nodelayer}
		\node [style=none] (0) at (0, 1) {};
		\node [style=none] (1) at (0, -1) {};
	\end{pgfonlayer}
	\begin{pgfonlayer}{edgelayer}
		\draw [cl] (0.center) to (1.center);
	\end{pgfonlayer}
\end{tikzpicture}
~:=~
\begin{tikzpicture}
	\begin{pgfonlayer}{nodelayer}
		\node [style=none] (0) at (0, 1.5) {};
		\node [style=none] (1) at (0, -0) {};
		\node [style=none] (2) at (0, -1.5) {};
	\end{pgfonlayer}
	\begin{pgfonlayer}{edgelayer}
		\draw [cr] (1.center) to (2.center);
		\draw [cr,-] (0.center) to (1.center);
	\end{pgfonlayer}
\end{tikzpicture}
\qqq
\overbrace{\underbrace{
\begin{tikzpicture}
	\begin{pgfonlayer}{nodelayer}
		\node [style=none] (0) at (-1.2, 1.6) {};
		\node [style=none] (1) at (1.2, 1.6) {};
		\node [style=bn] (2) at (0, -0) {$\theta$};
		\node [style=none] (3) at (-1.2, -1.6) {};
		\node [style=none] (4) at (1.2, -1.6) {};
	\end{pgfonlayer}
	\begin{pgfonlayer}{edgelayer}
		\draw [cr] (2) to (4.center);
		\draw [cr] (2) to (3.center);
		\draw [cl] (1.center) to (2);
		\draw [cl] (0.center) to (2);
	\end{pgfonlayer}
\end{tikzpicture}
}_n}^m
=~
  \overbrace{\underbrace{\begin{tikzpicture}
	\begin{pgfonlayer}{nodelayer}
		\node [style=none] (0) at (-0.6, 0.8) {};
		\node [style=none] (1) at (0, 1) {$\cdots$};
		\node [style=none] (2) at (0.6, 0.8) {};
		\node [style=gn] (3) at (0, -0) {$\theta$};
		\node [style=none] (4) at (-0.6, -0.8) {};
		\node [style=none] (5) at (0, -1) {$\cdots$};
		\node [style=none] (6) at (0.6, -0.8) {};
	\end{pgfonlayer}
	\begin{pgfonlayer}{edgelayer}
		\draw (2.center) to (3);
		\draw (3) to (6.center);
		\draw (0.center) to (3);
		\draw (3) to (4.center);
	\end{pgfonlayer}
\end{tikzpicture}}_n}^m
~=
  \overbrace{\underbrace{
\begin{tikzpicture}
	\begin{pgfonlayer}{nodelayer}
		\node [style=none] (0) at (-1.2, 1.6) {};
		\node [style=none] (1) at (1.2, 1.6) {};
		\node [style=rn] (2) at (0, -0) {$\theta$};
		\node [style=none] (3) at (-1.2, -1.6) {};
		\node [style=none] (4) at (1.2, -1.6) {};
	\end{pgfonlayer}
	\begin{pgfonlayer}{edgelayer}
		\draw [cl] (2) to (4.center);
		\draw [cl] (2) to (3.center);
		\draw [cr] (1.center) to (2);
		\draw [cr] (0.center) to (2);
	\end{pgfonlayer}
\end{tikzpicture}}_n}^m
\end{gather}

The above rules demonstrate the use of the $\identity[cr]$ and $\identity[cl]$ gates as colour rotation gates. They can be thought of as the analogues of the Hadamard gate from $\rg$. They also allow any node of any one colour to be expressed in terms of the other two colours.
\begin{equation}
  \begin{tikzpicture}
	\begin{pgfonlayer}{nodelayer}
		\node [style=none] (0) at (0, 1) {};
		\node [style=none] (1) at (0, -1) {};
        \node [style=none] at (0.5, 0) {};
        \node [style=none] at (-0.5, 0) {};
	\end{pgfonlayer}
	\begin{pgfonlayer}{edgelayer}
		\draw [ytick] (0.center) to (1.center);
	\end{pgfonlayer}
\end{tikzpicture}
~:=~
\begin{tikzpicture}
	\begin{pgfonlayer}{nodelayer}
		\node [style=none] (0) at (-1, 1) {};
		\node [style=gn] (1) at (0.5, 0.5) {};
		\node [style=rn] (2) at (-0.5, -0.5) {};
		\node [style=none] (3) at (1, -1) {};
	\end{pgfonlayer}
	\begin{pgfonlayer}{edgelayer}
		\draw (1) to (2);
		\draw (1) to (3.center);
		\draw (0.center) to (2);
	\end{pgfonlayer}
\end{tikzpicture}
~=~
\begin{tikzpicture}
	\begin{pgfonlayer}{nodelayer}
		\node [style=none] (0) at (-1, 1) {};
		\node [style=rn] (1) at (0.5, 0.5) {};
		\node [style=gn] (2) at (-0.5, -0.5) {};
		\node [style=none] (3) at (1, -1) {};
	\end{pgfonlayer}
	\begin{pgfonlayer}{edgelayer}
		\draw (1) to (2);
		\draw (1) to (3.center);
		\draw (0.center) to (2);
	\end{pgfonlayer}
\end{tikzpicture}
~=~
  \begin{tikzpicture}
	\begin{pgfonlayer}{nodelayer}
		\node [style=none] (0) at (0, 1) {};
		\node [style=rn] (1) at (0, -0) {$2$};
		\node [style=none] (2) at (0, -1) {};
	\end{pgfonlayer}
	\begin{pgfonlayer}{edgelayer}
		\draw (0.center) to (1);
		\draw (1) to (2.center);
	\end{pgfonlayer}
\end{tikzpicture}
\end{equation}
\begin{equation}
  \begin{tikzpicture}
	\begin{pgfonlayer}{nodelayer}
		\node [style=none] (0) at (0, 1) {};
		\node [style=none] (1) at (0, -1) {};
        \node [style=none] at (0.5, 0) {};
        \node [style=none] at (-0.5, 0) {};
	\end{pgfonlayer}
	\begin{pgfonlayer}{edgelayer}
		\draw [ctick] (0.center) to (1.center);
	\end{pgfonlayer}
\end{tikzpicture}
~:=~
\begin{tikzpicture}
	\begin{pgfonlayer}{nodelayer}
		\node [style=none] (0) at (-1, 1) {};
		\node [style=bn] (1) at (0.5, 0.5) {};
		\node [style=gn] (2) at (-0.5, -0.5) {};
		\node [style=none] (3) at (1, -1) {};
	\end{pgfonlayer}
	\begin{pgfonlayer}{edgelayer}
		\draw (1) to (2);
		\draw (1) to (3.center);
		\draw (0.center) to (2);
	\end{pgfonlayer}
\end{tikzpicture}
~=~
\begin{tikzpicture}
	\begin{pgfonlayer}{nodelayer}
		\node [style=none] (0) at (-1, 1) {};
		\node [style=gn] (1) at (0.5, 0.5) {};
		\node [style=bn] (2) at (-0.5, -0.5) {};
		\node [style=none] (3) at (1, -1) {};
	\end{pgfonlayer}
	\begin{pgfonlayer}{edgelayer}
		\draw (1) to (2);
		\draw (1) to (3.center);
		\draw (0.center) to (2);
	\end{pgfonlayer}
\end{tikzpicture}
~=~
  \begin{tikzpicture}
	\begin{pgfonlayer}{nodelayer}
		\node [style=none] (0) at (0, 1) {};
		\node [style=gn] (1) at (0, -0) {$2$};
		\node [style=none] (2) at (0, -1) {};
	\end{pgfonlayer}
	\begin{pgfonlayer}{edgelayer}
		\draw (0.center) to (1);
		\draw (1) to (2.center);
	\end{pgfonlayer}
\end{tikzpicture}
\end{equation}
\begin{equation}
  \begin{tikzpicture}
	\begin{pgfonlayer}{nodelayer}
		\node [style=none] (0) at (0, 1) {};
		\node [style=none] (1) at (0, -1) {};
        \node [style=none] at (0.5, 0) {};
        \node [style=none] at (-0.5, 0) {};
	\end{pgfonlayer}
	\begin{pgfonlayer}{edgelayer}
		\draw [mtick] (0.center) to (1.center);
	\end{pgfonlayer}
\end{tikzpicture}
~:=~
\begin{tikzpicture}
	\begin{pgfonlayer}{nodelayer}
		\node [style=none] (0) at (-1, 1) {};
		\node [style=rn] (1) at (0.5, 0.5) {};
		\node [style=bn] (2) at (-0.5, -0.5) {};
		\node [style=none] (3) at (1, -1) {};
	\end{pgfonlayer}
	\begin{pgfonlayer}{edgelayer}
		\draw (1) to (2);
		\draw (1) to (3.center);
		\draw (0.center) to (2);
	\end{pgfonlayer}
\end{tikzpicture}
~=~
\begin{tikzpicture}
	\begin{pgfonlayer}{nodelayer}
		\node [style=none] (0) at (-1, 1) {};
		\node [style=bn] (1) at (0.5, 0.5) {};
		\node [style=rn] (2) at (-0.5, -0.5) {};
		\node [style=none] (3) at (1, -1) {};
	\end{pgfonlayer}
	\begin{pgfonlayer}{edgelayer}
		\draw (1) to (2);
		\draw (1) to (3.center);
		\draw (0.center) to (2);
	\end{pgfonlayer}
\end{tikzpicture}
~=~
  \begin{tikzpicture}
	\begin{pgfonlayer}{nodelayer}
		\node [style=none] (0) at (0, 1) {};
		\node [style=bn] (1) at (0, -0) {$2$};
		\node [style=none] (2) at (0, -1) {};
	\end{pgfonlayer}
	\begin{pgfonlayer}{edgelayer}
		\draw (0.center) to (1);
		\draw (1) to (2.center);
	\end{pgfonlayer}
\end{tikzpicture}
\end{equation}

We define these convenient maps, termed dualizers, which are used to flip the direction of arrows. The colour yellow was chosen for the red/green dualizer as it is the additive colour combination of red and green, and similarly for cyan and magenta. When Equation~\ref{eq:cyc} is used to permute colours of a diagram, the dualizers must also be permuted, with accordance to their decomposition into red, green and blue nodes.
The dualizer notation is inspired by the notation used for the dualizer in ~\cite{coecke2011interacting} in the case of a pair of complementary observables with incompatible compact structures. In $\rgb$, we have three pairs of complementary observables with incompatible compact structures, hence we have three dualizers. In $\rg$, the pair of complementary observables has compatible compact structure, so the dualizer reduces to the identity:
\begin{equation}
  \begin{tikzpicture}
	\begin{pgfonlayer}{nodelayer}
		\node [style=none] (0) at (-0.7, 0.5) {};
		\node [style=ro] (1) at (0.4, 0.3) {};
		\node [style=go] (2) at (-0.4, -0.3) {};
		\node [style=none] (3) at (0.7, -0.5) {};
	\end{pgfonlayer}
	\begin{pgfonlayer}{edgelayer}
		\draw (0.center) to (2);
		\draw (1) to (2);
		\draw (1) to (3.center);
	\end{pgfonlayer}
\end{tikzpicture}
~=~
\begin{tikzpicture}
	\begin{pgfonlayer}{nodelayer}
		\node [style=none] (0) at (-0.7, 0.5) {};
		\node [style=go] (1) at (0.4, 0.3) {};
		\node [style=go] (2) at (-0.4, -0.3) {};
		\node [style=none] (3) at (0.7, -0.5) {};
	\end{pgfonlayer}
	\begin{pgfonlayer}{edgelayer}
		\draw (0.center) to (2);
		\draw (1) to (2);
		\draw (1) to (3.center);
	\end{pgfonlayer}
\end{tikzpicture}
~=~
\rot[]{go}
~=~
\identity
\end{equation}

One can notice that the rules which held for the red and green nodes in $\rg$ don't translate directly to the red and green dots of $\rgb$. This is made clearer by the interpretation in $\stab$ given below. A translation which does preserve this interpretation is also given further down.

The rules of $\rgb$ talk a lot about the rotation group of the Bloch sphere, specifically an octahedral subgroup of it. In fact, it is in a sense ``complete'' for the octahedral group, as demonstrated by the following result

\begin{prop}
  The octahedral group $O$ embeds faithfully into $\hom_{\rgb}(*,*)$. It is 
  isomorphic to the group generated by $\rot[$1$]{rn}$, $\rot[$1$]{gn}$ and $\rot[$1$]{bn}$
  \label{prop:octa}
\end{prop}

\begin{proof}
  The group $O$ is of abstract group type $S_4$ and can be given the following standard presentation:
\begin{align*}
O \cong S_4 \cong \langle \tau_1, \tau_2, \tau_3 | &\tau_1^2 = \tau_2^2 = \tau_3^2 = e, \tau_1\tau_3 = \tau_3\tau_1,\\
&\tau_1\tau_2\tau_1 = \tau_2\tau_1\tau_2, \tau_2\tau_3\tau_2 = \tau_3\tau_2\tau_3 \rangle
\end{align*}

If we look at the subcategory generated by $\rot[1]{rn}$, $\rot[1]{gn}$, $\rot[1]{bn}$ in $\hom_\rgb(*,*)$, we can see that (at least) the following relations hold on it:

\begin{align*}
G = \langle \sigma_r, \sigma_g, \sigma_b | \sigma_r^4 = \sigma_g^4 = \sigma_b^4 = \sigma_r^2 \sigma_g^2\sigma_b^2= e,\sigma_g\sigma_r = \sigma_b\sigma_g = \sigma_r\sigma_b\rangle
\end{align*}
where $\sigma_r$, $\sigma_g$, $\sigma_b$ represent $\rot[1]{rn}$, $\rot[1]{gn}$, $\rot[1]{bn}$ respectively and the group law is given by vertical composition---where reading from up to down is the same as reading from right to left. Note that it is clear from this presentation that $G$ is a quotient of the group $C_4*C_4*C_4$. We denote the quotient map $q:C_4*C_4*C_4\to G$ and will make use of it later.

We can then define the following group homomorphisms:
\begin{equation}
  f:O\to G =
  \begin{cases}
    \tau_1 &\mapsto \sigma_b\sigma_r^2\\
    \tau_2 &\mapsto \sigma_b^2\sigma_g\\
    \tau_3 &\mapsto \sigma_g\sigma_r\sigma_g
  \end{cases}
\end{equation}
\begin{equation}
  g:G\to O =
  \begin{cases}
    \sigma_r &\mapsto \tau_1\tau_2\tau_3\\
    \sigma_g &\mapsto \tau_3\tau_1\tau_2\\
    \sigma_b &\mapsto \tau_1\tau_2\tau_3\tau_1\tau_2
  \end{cases}
\end{equation}
After checking that these are indeed group homomorphisms, we can also determine that $f$ and $g$ are inverses, and thus $G \cong O$.

\end{proof}

%

\subsection{Some Derivable equations}

The equations below can be derived from the axioms of $\rgb$ given above. These equations can often be useful when wanting to demonstrates some more complex equalities in $\rgb$

The following equation is perhaps a more convenient form of the bialgebra law.
\begin{equation}
  \begin{tikzpicture}
	\begin{pgfonlayer}{nodelayer}
		\node [style=none] (0) at (-0.5, 1.5) {};
		\node [style=none] (1) at (0.5, 1.5) {};
		\node [style=gn] (2) at (-0.5, 0.5) {};
		\node [style=gn] (3) at (0.5, 0.5) {};
		\node [style=rn] (4) at (-0.5, -0.5) {};
		\node [style=rn] (5) at (0.5, -0.5) {};
		\node [style=none] (6) at (-0.5, -1.5) {};
		\node [style=none] (7) at (0.5, -1.5) {};
	\end{pgfonlayer}
	\begin{pgfonlayer}{edgelayer}
		\draw (2) to (5);
		\draw (2) to (4);
		\draw (4) to (6.center);
		\draw (3) to (5);
		\draw (0.center) to (2);
		\draw (3) to (4);
		\draw (5) to (7.center);
		\draw (1.center) to (3);
	\end{pgfonlayer}
\end{tikzpicture}
~=
\begin{tikzpicture}
	\begin{pgfonlayer}{nodelayer}
		\node [style=none] (0) at (-0.8, 1.1) {};
		\node [style=none] (1) at (0.8, 1.1) {};
		\node [style=rn] (2) at (0, 0.5) {};
		\node [style=bn] (4) at (0, -0.5) {$1$};
		\node [style=none] (5) at (-0.8, -1.1) {};
		\node [style=none] (6) at (0.8, -1.1) {};
	\end{pgfonlayer}
	\begin{pgfonlayer}{edgelayer}
		\draw (4) to (6.center);
		\draw (4) to (5.center);
		\draw (0.center) to (2);
		\draw (1.center) to (2);
		\draw (2) to (4);
	\end{pgfonlayer}
\end{tikzpicture}
\end{equation}

The following equation is a version of the Hopf law. It shows that in fact, all the bialgebras in $\rg$ that we had defined previously are also Hopf algebras where the antipode is the identity~\cite{kassel1995quantum}.
\begin{equation}
  \begin{tikzpicture}
	\begin{pgfonlayer}{nodelayer}
		\node [style=none] (0) at (0, 1.5) {};
		\node [style=gn] (1) at (0, 0.5) {};
		\node [style=rn] (2) at (0, -0.5) {$3$};
		\node [style=none] (3) at (0, -1.5) {};
	\end{pgfonlayer}
	\begin{pgfonlayer}{edgelayer}
		\draw (2) to (3.center);
		\draw (0.center) to (1);
		\draw[bend right=45] (1) to (2);
		\draw[bend left=45] (1) to (2);
	\end{pgfonlayer}
\end{tikzpicture}
~=~
\begin{tikzpicture}
	\begin{pgfonlayer}{nodelayer}
		\node [style=none] (0) at (0, 1.5) {};
		\node [style=gn] (1) at (0, 0.5) {};
		\node [style=bn] (2) at (0, -0.5) {};
		\node [style=none] (3) at (0, -1.5) {};
	\end{pgfonlayer}
	\begin{pgfonlayer}{edgelayer}
		\draw (2) to (3.center);
		\draw (0.center) to (1);
	\end{pgfonlayer}
\end{tikzpicture}
\end{equation}

The following equation shows that colour rotation in one direction is the inverse of colour rotation in the other direction.
\begin{equation}
  \begin{tikzpicture}
	\begin{pgfonlayer}{nodelayer}
		\node [style=none] (0) at (0, 1.25) {};
		\node [style=none] (1) at (0, -0) {};
		\node [style=none] (2) at (0, -1.25) {};
	\end{pgfonlayer}
	\begin{pgfonlayer}{edgelayer}
		\draw [cl] (1.center) to (2.center);
		\draw [cr,-] (0.center) to (1.center);
	\end{pgfonlayer}
\end{tikzpicture}
~=~
\identity
~=~
  \begin{tikzpicture}
	\begin{pgfonlayer}{nodelayer}
		\node [style=none] (0) at (0, 1.25) {};
		\node [style=none] (1) at (0, -0) {};
		\node [style=none] (2) at (0, -1.25) {};
	\end{pgfonlayer}
	\begin{pgfonlayer}{edgelayer}
		\draw [cr] (1.center) to (2.center);
		\draw [cl,-] (0.center) to (1.center);
	\end{pgfonlayer}
\end{tikzpicture}
\end{equation}

Additionally, the colour rotation rules can also be used to prove the following result, akin to Equation~\ref{eq:rgcyc} for the dichromatic calculus.
\begin{equation}
  \text{All rules hold under even permutations of colours.}
  \label{eq:cyc}
\end{equation}

The following equation demonstrates a convenient way to write a node of a colour in terms of nodes of the two other colours.
\begin{equation}
  \overbrace{\underbrace{
\begin{tikzpicture}
	\begin{pgfonlayer}{nodelayer}
		\node [style=none] (0) at (-1.25, 2.5) {};
		\node [style=none] (1) at (1.25, 2.5) {};
		\node [style=rn] (2) at (-1.25, 1.25) {$3$};
		\node [style=rn] (3) at (1.25, 1.25) {$3$};
		\node [style=bn] (4) at (0, -0) {$\theta-m+n$};
		\node [style=rn] (5) at (-1.25, -1.25) {$1$};
		\node [style=rn] (6) at (1.25, -1.25) {$1$};
		\node [style=none] (7) at (-1.25, -2.5) {};
		\node [style=none] (8) at (1.25, -2.5) {};
		\node [style=none] (9) at (0, 1.25) {$\cdots$};
		\node [style=none] (10) at (0, -1.25) {$\cdots$};
	\end{pgfonlayer}
	\begin{pgfonlayer}{edgelayer}
		\draw (1.center) to (3);
		\draw (4) to (6);
		\draw (2) to (4);
		\draw (0.center) to (2);
		\draw (3) to (4);
		\draw (5) to (7.center);
		\draw (4) to (5);
		\draw (6) to (8.center);
	\end{pgfonlayer}
\end{tikzpicture}
}_n}^m
~=
  \overbrace{\underbrace{
\begin{tikzpicture}
	\begin{pgfonlayer}{nodelayer}
		\node [style=none] (0) at (-0.75, 1) {};
		\node [style=none] (1) at (0, 1) {$\cdots$};
		\node [style=none] (2) at (0.75, 1) {};
		\node [style=gn] (3) at (0, -0) {$\theta$};
		\node [style=none] (4) at (-0.75, -1) {};
		\node [style=none] (5) at (0, -1) {$\cdots$};
		\node [style=none] (6) at (0.75, -1) {};
	\end{pgfonlayer}
	\begin{pgfonlayer}{edgelayer}
		\draw (2.center) to (3);
		\draw (3) to (6.center);
		\draw (0.center) to (3);
		\draw (3) to (4.center);
	\end{pgfonlayer}
\end{tikzpicture}
}_n}^m
=~
  \overbrace{\underbrace{
\begin{tikzpicture}
	\begin{pgfonlayer}{nodelayer}
		\node [style=none] (0) at (-1.25, 2.5) {};
		\node [style=none] (1) at (1.25, 2.5) {};
		\node [style=bn] (2) at (-1.25, 1.25) {$1$};
		\node [style=bn] (3) at (1.25, 1.25) {$1$};
		\node [style=rn] (4) at (0, -0) {$\theta+m-n$};
		\node [style=bn] (5) at (-1.25, -1.25) {$3$};
		\node [style=bn] (6) at (1.25, -1.25) {$3$};
		\node [style=none] (7) at (-1.25, -2.5) {};
		\node [style=none] (8) at (1.25, -2.5) {};
		\node [style=none] (9) at (0, 1.25) {$\cdots$};
		\node [style=none] (10) at (0, -1.25) {$\cdots$};
	\end{pgfonlayer}
	\begin{pgfonlayer}{edgelayer}
		\draw (1.center) to (3);
		\draw (4) to (6);
		\draw (2) to (4);
		\draw (0.center) to (2);
		\draw (3) to (4);
		\draw (5) to (7.center);
		\draw (4) to (5);
		\draw (6) to (8.center);
	\end{pgfonlayer}
\end{tikzpicture}
}_n}^m
\end{equation}

The following equation shows how dualizers can be used to invert arrows in a digraph.
\begin{equation}
\begin{tikzpicture}
	\begin{pgfonlayer}{nodelayer}
		\node [style=none] (0) at (-1.6, 0.8) {};
		\node [style=none] (1) at (-1, 0.8) {$\cdots$};
		\node [style=none] (2) at (-0.4, 0.8) {};
		\node [style=none] (3) at (0.4, 0.8) {};
		\node [style=none] (4) at (1, 0.8) {$\cdots$};
		\node [style=none] (5) at (1.6, 0.8) {};
		\node [style=gn] (6) at (-1, -0) {};
		\node [style=rn] (7) at (1, -0) {};
		\node [style=none] (8) at (-1.6, -0.8) {};
		\node [style=none] (9) at (-1, -0.8) {$\cdots$};
		\node [style=none] (10) at (-0.4, -0.8) {};
		\node [style=none] (11) at (0.4, -0.8) {};
		\node [style=none] (12) at (1, -0.8) {$\cdots$};
		\node [style=none] (13) at (1.6, -0.8) {};
	\end{pgfonlayer}
	\begin{pgfonlayer}{edgelayer}
		\draw (7) to (13.center);
		\draw (6) to (10.center);
		\draw (0.center) to (6);
		\draw (6) to (7);
		\draw (5.center) to (7);
		\draw (7) to (11.center);
		\draw (3.center) to (7);
		\draw (6) to (8.center);
		\draw (2.center) to (6);
	\end{pgfonlayer}
\end{tikzpicture}
~=~
\begin{tikzpicture}
	\begin{pgfonlayer}{nodelayer}
		\node [style=none] (0) at (-1.6, 0.8) {};
		\node [style=none] (1) at (-1, 0.8) {$\cdots$};
		\node [style=none] (2) at (-0.4, 0.8) {};
		\node [style=none] (3) at (0.4, 0.8) {};
		\node [style=none] (4) at (1, 0.8) {$\cdots$};
		\node [style=none] (5) at (1.6, 0.8) {};
		\node [style=gn] (6) at (-1, -0) {};
		\node [style=rn] (7) at (1, -0) {};
		\node [style=none] (8) at (-1.6, -0.8) {};
		\node [style=none] (9) at (-1, -0.8) {$\cdots$};
		\node [style=none] (10) at (-0.4, -0.8) {};
		\node [style=none] (11) at (0.4, -0.8) {};
		\node [style=none] (12) at (1, -0.8) {$\cdots$};
		\node [style=none] (13) at (1.6, -0.8) {};
	\end{pgfonlayer}
	\begin{pgfonlayer}{edgelayer}
		\draw (7) to (13.center);
		\draw (6) to (10.center);
		\draw (0.center) to (6);
		\draw [ytick] (7) to (6);
		\draw (5.center) to (7);
		\draw (7) to (11.center);
		\draw (3.center) to (7);
		\draw (6) to (8.center);
		\draw (2.center) to (6);
	\end{pgfonlayer}
\end{tikzpicture}
\end{equation}

The following equation shows that two dualizers of the same colour annihilate.
\begin{equation}
\begin{tikzpicture}
	\begin{pgfonlayer}{nodelayer}
		\node [style=none] (0) at (0, 1) {};
		\node [style=none] (1) at (0, -0) {};
		\node [style=none] (2) at (0, -1) {};
        \node [style=none] at (0.5, 0) {};
        \node [style=none] at (-0.5, 0) {};
	\end{pgfonlayer}
	\begin{pgfonlayer}{edgelayer}
		\draw [ytick,-] (0.center) to (1.center);
		\draw [ytick] (1.center) to (2.center);
	\end{pgfonlayer}
\end{tikzpicture}
=~
\identity
\end{equation}

The following equation shows that three heterochromatic dualizers annihilate.
\begin{equation}
  \begin{tikzpicture}
	\begin{pgfonlayer}{nodelayer}
		\node [style=none] (0) at (0, 1.5) {};
		\node [style=none] (1) at (0, 0.5) {};
		\node [style=none] (2) at (0, -0.5) {};
		\node [style=none] (3) at (0, -1.5) {};
        \node [style=none] at (0.5, 0) {};
        \node [style=none] at (-0.5, 0) {};
	\end{pgfonlayer}
	\begin{pgfonlayer}{edgelayer}
		\draw [ytick,-] (0.center) to (1.center);
		\draw [ctick,-] (1.center) to (2.center);
		\draw [mtick] (2.center) to (3.center);
	\end{pgfonlayer}
\end{tikzpicture}
=~
\identity
~=
  \begin{tikzpicture}
	\begin{pgfonlayer}{nodelayer}
		\node [style=none] (0) at (0, 1.5) {};
		\node [style=none] (1) at (0, 0.5) {};
		\node [style=none] (2) at (0, -0.5) {};
		\node [style=none] (3) at (0, -1.5) {};
        \node [style=none] at (0.5, 0) {};
        \node [style=none] at (-0.5, 0) {};
	\end{pgfonlayer}
	\begin{pgfonlayer}{edgelayer}
		\draw [mtick,-] (0.center) to (1.center);
		\draw [ctick,-] (1.center) to (2.center);
		\draw [ytick] (2.center) to (3.center);
	\end{pgfonlayer}
\end{tikzpicture}
\end{equation}

The remaining equations are useful equations about dualizers:

\begin{gather}
  \begin{tikzpicture}
	\begin{pgfonlayer}{nodelayer}
		\node [style=gn] (0) at (0, 0.5) {};
		\node [style=none] (1) at (0, -0.5) {};
	\end{pgfonlayer}
	\begin{pgfonlayer}{edgelayer}
		\draw [ytick] (0) to (1.center);
	\end{pgfonlayer}
\end{tikzpicture}
~=~
\begin{tikzpicture}
	\begin{pgfonlayer}{nodelayer}
		\node [style=gn] (0) at (0, 0.5) {};
		\node [style=none] (1) at (0, -0.5) {};
	\end{pgfonlayer}
	\begin{pgfonlayer}{edgelayer}
		\draw (0) to (1.center);
	\end{pgfonlayer}
\end{tikzpicture}
\qqq
  \begin{tikzpicture}
	\begin{pgfonlayer}{nodelayer}
		\node [style=rn] (0) at (0, 0.5) {};
		\node [style=none] (1) at (0, -0.5) {};
	\end{pgfonlayer}
	\begin{pgfonlayer}{edgelayer}
		\draw [ytick] (0) to (1.center);
	\end{pgfonlayer}
\end{tikzpicture}
~=~
\begin{tikzpicture}
	\begin{pgfonlayer}{nodelayer}
		\node [style=rn] (0) at (0, 0.5) {$2$};
		\node [style=none] (1) at (0, -0.5) {};
	\end{pgfonlayer}
	\begin{pgfonlayer}{edgelayer}
		\draw (0) to (1.center);
	\end{pgfonlayer}
\end{tikzpicture}
\qqq
  \begin{tikzpicture}
	\begin{pgfonlayer}{nodelayer}
		\node [style=bn] (0) at (0, 0.5) {};
		\node [style=none] (1) at (0, -0.5) {};
	\end{pgfonlayer}
	\begin{pgfonlayer}{edgelayer}
		\draw [ytick] (0) to (1.center);
	\end{pgfonlayer}
\end{tikzpicture}
~=~
\begin{tikzpicture}
	\begin{pgfonlayer}{nodelayer}
		\node [style=bn] (0) at (0, 0.5) {$2$};
		\node [style=none] (1) at (0, -0.5) {};
	\end{pgfonlayer}
	\begin{pgfonlayer}{edgelayer}
		\draw (0) to (1.center);
	\end{pgfonlayer}
\end{tikzpicture}
\\
\begin{tikzpicture}
	\begin{pgfonlayer}{nodelayer}
		\node [style=none] (0) at (-0.5, 1) {};
		\node [style=none] (1) at (0, 1) {$\cdots$};
		\node [style=none] (2) at (0.5, 1) {};
		\node [style=gn] (3) at (0, -0) {};
		\node [style=none] (4) at (-0.5, -1) {};
		\node [style=none] (5) at (0, -1) {$\cdots$};
		\node [style=none] (6) at (0.5, -1) {};
	\end{pgfonlayer}
	\begin{pgfonlayer}{edgelayer}
		\draw [ytick] (2.center) to (3);
		\draw [ytick] (3) to (6.center);
		\draw [ytick] (0.center) to (3);
		\draw [ytick] (3) to (4.center);
	\end{pgfonlayer}
\end{tikzpicture}
~=~
\begin{tikzpicture}
	\begin{pgfonlayer}{nodelayer}
		\node [style=none] (0) at (-0.5, 1) {};
		\node [style=none] (1) at (0, 1) {$\cdots$};
		\node [style=none] (2) at (0.5, 1) {};
		\node [style=gn] (3) at (0, -0) {};
		\node [style=none] (4) at (-0.5, -1) {};
		\node [style=none] (5) at (0, -1) {$\cdots$};
		\node [style=none] (6) at (0.5, -1) {};
	\end{pgfonlayer}
	\begin{pgfonlayer}{edgelayer}
		\draw (2.center) to (3);
		\draw (3) to (6.center);
		\draw (0.center) to (3);
		\draw (3) to (4.center);
	\end{pgfonlayer}
\end{tikzpicture}
\qqq
    \begin{tikzpicture}
	\begin{pgfonlayer}{nodelayer}
		\node [style=none] (0) at (0, 1) {};
		\node [style=rn] (1) at (0, -0) {};
		\node [style=none] (2) at (-0.5, -1) {};
		\node [style=none] (3) at (0.5, -1) {};
	\end{pgfonlayer}
	\begin{pgfonlayer}{edgelayer}
		\draw [ytick] (1) to (2.center);
		\draw (1) to (3.center);
		\draw (0.center) to (1);
	\end{pgfonlayer}
\end{tikzpicture}
~=~
\begin{tikzpicture}
	\begin{pgfonlayer}{nodelayer}
		\node [style=none] (0) at (0, 1) {};
		\node [style=rn] (1) at (0, -0) {};
		\node [style=none] (2) at (-0.5, -1) {};
		\node [style=none] (3) at (0.5, -1) {};
	\end{pgfonlayer}
	\begin{pgfonlayer}{edgelayer}
		\draw (1) to (2.center);
		\draw [ytick] (1) to (3.center);
		\draw (0.center) to (1);
	\end{pgfonlayer}
\end{tikzpicture}
~=~
\begin{tikzpicture}
	\begin{pgfonlayer}{nodelayer}
		\node [style=none] (0) at (0, 1) {};
		\node [style=rn] (1) at (0, -0) {};
		\node [style=none] (2) at (-0.5, -1) {};
		\node [style=none] (3) at (0.5, -1) {};
	\end{pgfonlayer}
	\begin{pgfonlayer}{edgelayer}
		\draw (1) to (2.center);
		\draw (1) to (3.center);
		\draw [ytick] (0.center) to (1);
	\end{pgfonlayer}
\end{tikzpicture}
~=~
\begin{tikzpicture}
	\begin{pgfonlayer}{nodelayer}
		\node [style=none] (0) at (0, 1) {};
		\node [style=rn] (1) at (0, -0) {};
		\node [style=none] (2) at (-0.5, -1) {};
		\node [style=none] (3) at (0.5, -1) {};
	\end{pgfonlayer}
	\begin{pgfonlayer}{edgelayer}
		\draw [ytick] (1) to (2.center);
		\draw [ytick] (1) to (3.center);
		\draw [ytick] (0.center) to (1);
	\end{pgfonlayer}
\end{tikzpicture}
\qqq
\begin{tikzpicture}
	\begin{pgfonlayer}{nodelayer}
		\node [style=none] (0) at (-0.5, 1) {};
		\node [style=none] (1) at (0, 1) {$\cdots$};
		\node [style=none] (2) at (0.5, 1) {};
		\node [style=bn] (3) at (0, -0) {};
		\node [style=none] (4) at (-0.5, -1) {};
		\node [style=none] (5) at (0, -1) {$\cdots$};
		\node [style=none] (6) at (0.5, -1) {};
	\end{pgfonlayer}
	\begin{pgfonlayer}{edgelayer}
		\draw [ytick] (2.center) to (3);
		\draw [ytick] (3) to (6.center);
		\draw [ytick] (0.center) to (3);
		\draw [ytick] (3) to (4.center);
	\end{pgfonlayer}
\end{tikzpicture}
~=~
\begin{tikzpicture}
	\begin{pgfonlayer}{nodelayer}
		\node [style=none] (0) at (-0.5, 1) {};
		\node [style=none] (1) at (0, 1) {$\cdots$};
		\node [style=none] (2) at (0.5, 1) {};
		\node [style=bn] (3) at (0, -0) {};
		\node [style=none] (4) at (-0.5, -1) {};
		\node [style=none] (5) at (0, -1) {$\cdots$};
		\node [style=none] (6) at (0.5, -1) {};
	\end{pgfonlayer}
	\begin{pgfonlayer}{edgelayer}
		\draw (2.center) to (3);
		\draw (3) to (6.center);
		\draw (0.center) to (3);
		\draw (3) to (4.center);
	\end{pgfonlayer}
\end{tikzpicture}
\end{gather}

\subsection{$\dag$ Structure}
In a similar fashion as in $\rg$, the symmetric monoidal category $\rgb$ can further be made into a \dsmc{} by having $\dag$ act on the generators like so:
\[\begin{array}{c@{\qquad}c@{\qquad}c@{\qquad}c@{\qquad}c}
  \dagof{\unit{rn}}=~\counit{rn}
&
  \dagof{\rot{rn}}=~\rot[$-\theta$]{rn}
&
  \dagof{\counit{rn}}=~\unit{rn}
&
  \dagof{\mul{rn}}=\comul{rn}
&
  \dagof{\comul{rn}}=\mul{rn}
\\
  \dagof{\unit{gn}}=~\counit{gn}
&
  \dagof{\rot{gn}}=~\rot[$-\theta$]{gn}
&
  \dagof{\counit{gn}}=~\unit{gn}
&
  \dagof{\mul{gn}}=\comul{gn}
&
  \dagof{\comul{gn}}=\mul{gn}
\\
  \dagof{\unit{bn}}=~\counit{bn}
&
  \dagof{\rot{bn}}=~\rot[$-\theta$]{bn}
&
  \dagof{\counit{bn}}=~\unit{bn}
&
  \dagof{\mul{bn}}=\comul{bn} 
&
  \dagof{\comul{bn}}=\mul{bn}
\end{array}
\]
where functoriality of $\dagof{\cdot}$ is guaranteed by Rule~\eqref{eq:rgdag}

\begin{figure*} 
  \centering
  \includegraphics[scale=0.4]{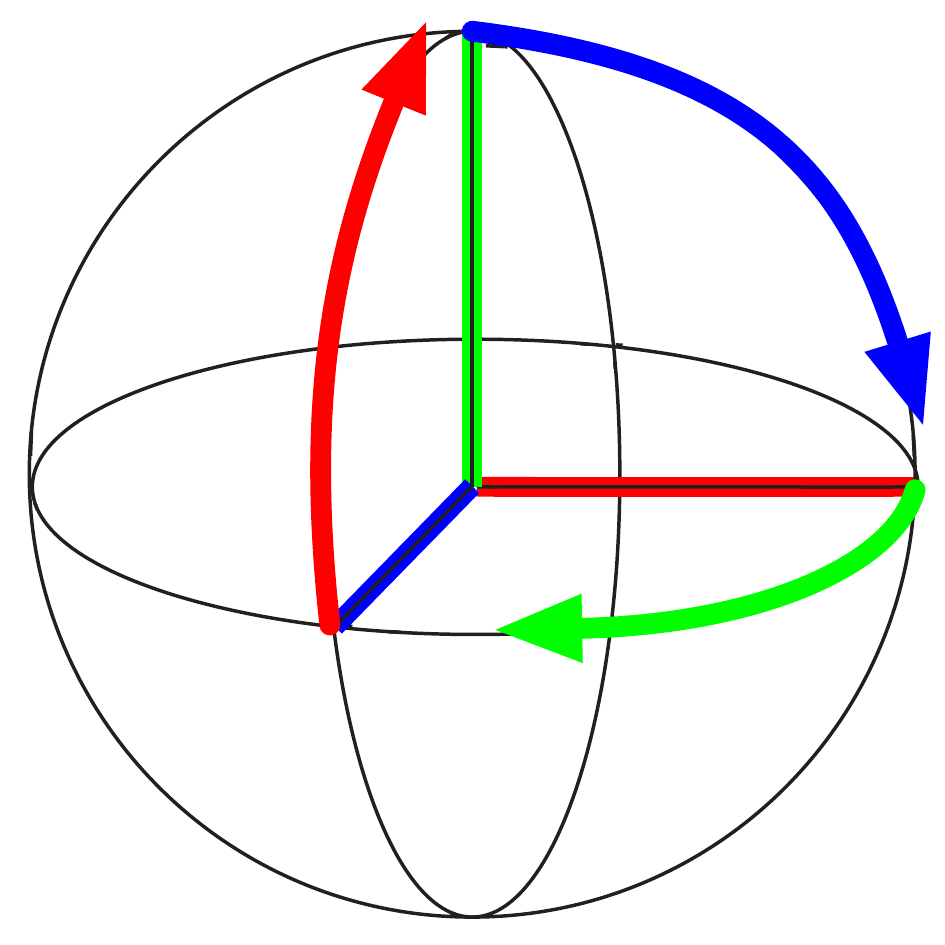}
  \caption{$\rgb$ Bloch sphere}
  \label{fig:rgbbloch}
\end{figure*}

\subsection{$\rgb$ interpretation}
In a similar manner to what we did for $\rg$, we provide an interpretation $\rgbtostab{\cdot}:\rgb\to\stab$ of the morphisms in $\rgb$

\begin{gather*}
    \rgbtostab{\unit{gn}} ~=~  \ket{+}
    \qqq
    \rgbtostab{\rot{gn}} ~=~ \ket0\bra0~+~ e^{i\theta\frac{\pi}{2}}\ket1\bra1
    \qqq
    \rgbtostab{\counit{gn}} ~=~ \bra{+}
  \\
    \rgbtostab{\comul{gn}} ~=~ \ket{00}\bra{0}~+~\ket{11}\bra{1}
    \qqq
    \rgbtostab{\mul{gn}} ~=~ \ket{0}\bra{00}~+~\ket{1}\bra{11}
\\
    \rgbtostab{\unit{rn}} ~=~ \ket{i}
  \qqq
  \rgbtostab{\rot{rn}} ~=~ \ket+\bra+ ~+~ e^{i\theta\frac{\pi}{2}}\ket-\bra-
  \qqq
    \rgbtostab{\counit{rn}} ~=~ \bra{i}
  \\
    \rgbtostab{\comul{rn}} ~=~ \ket{++}\bra{+} ~-~i\ket{--}\bra{-}
    \qqq
    \rgbtostab{\mul{rn}} ~=~ \ket{+}\bra{++}~+~i\ket{-}\bra{--}
\end{gather*}
\begin{gather*}
    \rgbtostab{\unit{bn}}~=~ \ket{0}
    \qqq
  \rgbtostab{\rot{bn}} ~=~ \ket i\bra i ~+~ e^{i\theta\frac{\pi}{2}}\ket{-i}\bra{-i}
    \qqq
    \rgbtostab{\counit{bn}} ~=~ \bra{0}
  \\
    \rgbtostab{\comul{bn}} ~=~ \ket{ii}\bra{i} ~+~\ket{-i -i}\bra{-i}
    \qqq
    \rgbtostab{\mul{bn}} ~=~ \ket{i}\bra{ii} ~+~\ket{-i}\bra{-i -i}
\end{gather*}

\begin{prop}
  $\rgbtostab{\cdot}$ is a symmetric monoidal $\dag$-functor.
\end{prop}

\begin{proof}
  See proof of Proposition~\ref{prop:rginter}.
\end{proof}

Informally, the difference between $\rg$ and $\rgb$ can be seen in Figure~\ref{fig:rgbbloch}. The rotational symmetry of the three colours can clearly be seen, and contrasted with Figure~\ref{fig:rgbloch}. For each observable, the straight segment represents the axis of rotation concretized by the $\rot{gn}$ morphism---it is also the line of values which are fixed by that morphism. The arrow shows the direction of positive rotation, and the arrow-tail is where the deleting point $\unit{gn}$ lies. 

One can notice that the red deleting point in Figure~\ref{fig:rgbbloch} is at a different location than the red deleting point in Figure~\ref{fig:rgbloch}. It is for this reason that---although the green generators of $\rg$ and $\rgb$ are interpreted as the same values in $\stab$---the red generators are mapped to different values. Essentially, this design choice was made so that Rule~\eqref{eq:cyc}, and the rules involving the colour changers can hold.

\section{RG to RGB translation}

We define a translation from dichromatic diagrams to trichromatic diagrams as a functor $\mathcal{T}:\rg\to\rgb$ by first defining $\mathcal{T}$ by its value on the generators of $\rg$ and then checking that equal diagrams of $\rg$ are equal under translation.
\[\begin{array}{@{}c@{~\quad}c@{~\quad}c@{~\quad}c@{~\quad}c@{}}
  \rgtorgbof{\counit{go}}=~\counit{gn}
& \rgtorgbof{\rot{go}}=~\rot{gn}
& \rgtorgbof{\unit{go}}=~\unit{gn}
&
  \rgtorgbof{\comul{go}}=\comul{gn}
&
  \rgtorgbof{\mul{go}}=\mul{gn}
\\
  \rgtorgbof{\counit{ro}}=~\counit[$3$]{rn}
&
\rgtorgbof{\rot{ro}}=~\rot{rn}
&
  \rgtorgbof{\unit{ro}}=~\unit[$1$]{rn}
&
  \rgtorgbof{\comul{ro}}=\comul[$1$]{rn}
&
  \rgtorgbof{\mul{ro}}=\mul[$3$]{rn}
\end{array}
\]

\begin{prop}
  $\rgtorgb$ is a functor.
\end{prop}

\begin{proof}
  It suffices to do a routine check that for all rules $f=g$ in $\rg$, we can prove $\rgtorgb{f}=\rgtorgb{g}$ in $\rgb$.
\end{proof}

\subsection{Translation preserves interpretation}
Finally, the crucial property of this translation is the following

\begin{prop}
  The following diagram commutes
  \begin{equation}
      \xymatrix{
  \rg
\ar[rr]^{\mathcal{T}}
\ar[dr]_{\rgtostab{\cdot}}
&&\rgb
\ar[dl]^{\rgbtostab{\cdot}}\\
&\stab}
  \end{equation}
\end{prop}

\begin{proof}
  This follows from a routine check that for each generator $\varphi:*^m\to*^n$ in $\rg$, $\rgtorgb{\rgbtostab{\varphi}}=\rgtostab{\varphi}$
\end{proof}

\begin{prop}
  \label{prop:supp}
  The supplementarity rule that we had mentioned earlier is provable if we take its translation into $\rgb$. That is,
  \begin{equation}
    \rgtorgbof{
\begin{tikzpicture}
	\begin{pgfonlayer}{nodelayer}
		\node [style=none] (0) at (0, 1.6) {};
		\node [style=go] (1) at (0, 0.6) {};
		\node [style=ro] (2) at (-0.8, -0) {$1$};
		\node [style=ro] (3) at (0.8, -0) {$1$};
		\node [style=go] (4) at (0, -0.6) {};
		\node [style=none] (5) at (0, -1.6) {};
	\end{pgfonlayer}
	\begin{pgfonlayer}{edgelayer}
		\draw (2) to (4);
		\draw (1) to (2);
		\draw (1) to (3);
		\draw (3) to (4);
		\draw (4) to (5.center);
		\draw (0.center) to (1);
	\end{pgfonlayer}
\end{tikzpicture}}
~=~
\rgtorgbof{\begin{tikzpicture}
	\begin{pgfonlayer}{nodelayer}
		\node [style=none] (0) at (0, 1.5) {};
		\node [style=go] (1) at (0, 0.5) {$2$};
		\node [style=go] (2) at (0, -0.5) {$2$};
		\node [style=none] (3) at (0, -1.5) {};
	\end{pgfonlayer}
	\begin{pgfonlayer}{edgelayer}
		\draw (0.center) to (1);
		\draw (2) to (3.center);
	\end{pgfonlayer}
\end{tikzpicture}}
  \end{equation}
\end{prop}

\begin{proof}
\begin{equation}
  \begin{split}
\rgtorgbof{\begin{tikzpicture}
	\begin{pgfonlayer}{nodelayer}
		\node [style=none] (0) at (0, 1.6) {};
		\node [style=go] (1) at (0, 0.6) {};
		\node [style=ro] (2) at (-0.8, -0) {$1$};
		\node [style=ro] (3) at (0.8, -0) {$1$};
		\node [style=go] (4) at (0, -0.6) {};
		\node [style=none] (5) at (0, -1.6) {};
	\end{pgfonlayer}
	\begin{pgfonlayer}{edgelayer}
		\draw (2) to (4);
		\draw (1) to (2);
		\draw (1) to (3);
		\draw (3) to (4);
		\draw (4) to (5.center);
		\draw (0.center) to (1);
	\end{pgfonlayer}
\end{tikzpicture}}
~=~
\begin{tikzpicture}
	\begin{pgfonlayer}{nodelayer}
		\node [style=none] (0) at (0, 1.6) {};
		\node [style=gn] (1) at (0, 0.6) {};
		\node [style=rn] (2) at (-0.8, -0) {$1$};
		\node [style=rn] (3) at (0.8, -0) {$1$};
		\node [style=gn] (4) at (0, -0.6) {};
		\node [style=none] (5) at (0, -1.6) {};
	\end{pgfonlayer}
	\begin{pgfonlayer}{edgelayer}
		\draw (2) to (4);
		\draw (1) to (2);
		\draw (1) to (3);
		\draw (3) to (4);
		\draw (4) to (5.center);
		\draw (0.center) to (1);
	\end{pgfonlayer}
\end{tikzpicture}
~=~
\begin{tikzpicture}
	\begin{pgfonlayer}{nodelayer}
		\node [style=none] (0) at (0, 1.6) {};
		\node [style=gn] (1) at (0, 0.6) {};
		\node [style=rn] (2) at (-0.8, -0) {$1$};
		\node [style=rn] (3) at (0.8, -0) {$1$};
		\node [style=gn] (4) at (0, -0.6) {};
		\node [style=none] (5) at (0, -1.6) {};
	\end{pgfonlayer}
	\begin{pgfonlayer}{edgelayer}
		\draw [ytick] (4) to (2);
		\draw (1) to (2);
		\draw (1) to (3);
		\draw [ytick] (4) to (3);
		\draw (4) to (5.center);
		\draw (0.center) to (1);
	\end{pgfonlayer}
\end{tikzpicture}
~=~
\begin{tikzpicture}
	\begin{pgfonlayer}{nodelayer}
		\node [style=none] (0) at (0, 1.6) {};
		\node [style=gn] (1) at (0, 0.6) {};
		\node [style=rn] (2) at (-0.8, -0) {$1$};
		\node [style=rn] (3) at (0.8, -0) {$1$};
		\node [style=gn] (4) at (0, -0.6) {};
		\node [style=none] (5) at (0, -1.6) {};
	\end{pgfonlayer}
	\begin{pgfonlayer}{edgelayer}
		\draw (4) to (2);
		\draw (1) to (2);
		\draw (1) to (3);
		\draw (4) to (3);
		\draw [ytick] (4) to (5.center);
		\draw (0.center) to (1);
	\end{pgfonlayer}
\end{tikzpicture}
~=~
\begin{tikzpicture}
	\begin{pgfonlayer}{nodelayer}
		\node [style=none] (0) at (-0.5, 2) {};
		\node [style=rn] (1) at (0.5, 1) {$1$};
		\node [style=rn] (2) at (-0.5, 0) {};
		\node [style=bn] (3) at (0.5, 0) {$1$};
		\node [style=gn] (4) at (-0.5, -1) {};
		\node [style=rn] (5) at (0.5, -1) {$1$};
		\node [style=none] (6) at (-0.5, -2) {};
	\end{pgfonlayer}
	\begin{pgfonlayer}{edgelayer}
		\draw (3) to (5);
		\draw (3) to (1);
		\draw (2) to (3);
		\draw [ytick] (4) to (6.center);
		\draw (0.center) to (2);
		\draw (4) to (2);
	\end{pgfonlayer}
\end{tikzpicture}
~=~
\begin{tikzpicture}
	\begin{pgfonlayer}{nodelayer}
		\node [style=none] (0) at (-0.5, 2) {};
		\node [style=rn] (1) at (0.5, 1) {$1$};
		\node [style=rn] (2) at (-0.5, 0) {};
		\node [style=bn] (3) at (0.5, 0) {$1$};
		\node [style=gn] (4) at (-0.5, -1) {};
		\node [style=rn] (5) at (0.5, -1) {$1$};
		\node [style=none] (6) at (-0.5, -2) {};
	\end{pgfonlayer}
	\begin{pgfonlayer}{edgelayer}
		\draw (3) to (5);
		\draw (3) to (1);
		\draw (2) to (3);
		\draw (4) to (6.center);
		\draw (0.center) to (2);
		\draw [ytick] (4) to (2);
	\end{pgfonlayer}
\end{tikzpicture}
~=~
\begin{tikzpicture}
	\begin{pgfonlayer}{nodelayer}
		\node [style=none] (0) at (-0.5, 2) {};
		\node [style=rn] (1) at (0.5, 1) {$1$};
		\node [style=rn] (2) at (-0.5, 0) {};
		\node [style=bn] (3) at (0.5, 0) {$1$};
		\node [style=rn] (5) at (0.5, -1) {$1$};
		\node [style=none] (6) at (-0.5, -2) {};
	\end{pgfonlayer}
	\begin{pgfonlayer}{edgelayer}
		\draw (3) to (5);
		\draw (3) to (1);
		\draw (2) to (3);
		\draw (2) to (6.center);
		\draw (0.center) to (2);
	\end{pgfonlayer}
\end{tikzpicture}
\\~=~
\begin{tikzpicture}
	\begin{pgfonlayer}{nodelayer}
		\node [style=none] (0) at (-0.5, 2) {};
		\node [style=rn] (1) at (0.5, 1) {};
		\node [style=rn] (2) at (-0.5, 0) {$1$};
		\node [style=gn] (3) at (0.5, 0) {};
		\node [style=rn] (5) at (0.5, -1) {};
		\node [style=none] (6) at (-0.5, -2) {};
	\end{pgfonlayer}
	\begin{pgfonlayer}{edgelayer}
		\draw (3) to (5);
		\draw (3) to (1);
		\draw (2) to (3);
		\draw (2) to (6.center);
		\draw (0.center) to (2);
	\end{pgfonlayer}
\end{tikzpicture}
~=~
\begin{tikzpicture}
	\begin{pgfonlayer}{nodelayer}
		\node [style=none] (0) at (-0.5, 1) {};
		\node [style=rn] (2) at (-0.5, 0) {$1$};
		\node [style=gn] (3) at (0.5, 0) {$2$};
		\node [style=none] (6) at (-0.5, -1) {};
	\end{pgfonlayer}
	\begin{pgfonlayer}{edgelayer}
		\draw (2) to (3);
		\draw (2) to (6.center);
		\draw (0.center) to (2);
	\end{pgfonlayer}
\end{tikzpicture}
~=~
\begin{tikzpicture}
	\begin{pgfonlayer}{nodelayer}
		\node [style=none] (0) at (-0, 1) {};
		\node [style=rn] (2) at (-0, 0) {};
		\node [style=gn] (3) at (1, 0) {};
        \node [style=bn] (4) at (-1, 0) {};
		\node [style=none] (6) at (-0, -1) {};
	\end{pgfonlayer}
	\begin{pgfonlayer}{edgelayer}
		\draw [ctick] (2) to (3);
		\draw (2) to (6.center);
		\draw (0.center) to (2);
        \draw (4) to (2);
	\end{pgfonlayer}
\end{tikzpicture}
~=~
\begin{tikzpicture}
	\begin{pgfonlayer}{nodelayer}
		\node [style=none] (0) at (-0, 1) {};
		\node [style=rn] (2) at (-0, 0) {};
		\node [style=gn] (3) at (1, 0) {};
        \node [style=bn] (4) at (-1, 0) {};
		\node [style=none] (6) at (-0, -1) {};
	\end{pgfonlayer}
	\begin{pgfonlayer}{edgelayer}
		\draw (2) to (3);
		\draw [ctick] (2) to (6.center);
		\draw [ctick] (0.center) to (2);
        \draw [ctick] (4) to (2);
	\end{pgfonlayer}
\end{tikzpicture}
~=~
\begin{tikzpicture}
	\begin{pgfonlayer}{nodelayer}
		\node [style=none] (0) at (1, 2) {};
		\node [style=gn] (2) at (-0, 0) {};
		\node [style=gn] (3) at (1, 1) {};
        \node [style=bn] (4) at (-1, 0) {};
        \node [style=gn] (5) at (1, -1) {};
		\node [style=none] (6) at (1, -2) {};
	\end{pgfonlayer}
	\begin{pgfonlayer}{edgelayer}
		\draw [ctick] (5) to (6.center);
		\draw [ctick] (0.center) to (3);
        \draw [ctick] (4) to (2);
	\end{pgfonlayer}
\end{tikzpicture}
~=~
\begin{tikzpicture}
	\begin{pgfonlayer}{nodelayer}
		\node [style=none] (0) at (0, 1.5) {};
		\node [style=gn] (1) at (0, 0.5) {$2$};
		\node [style=gn] (2) at (0, -0.5) {$2$};
		\node [style=none] (3) at (0, -1.5) {};
	\end{pgfonlayer}
	\begin{pgfonlayer}{edgelayer}
		\draw (0.center) to (1);
		\draw (2) to (3.center);
	\end{pgfonlayer}
\end{tikzpicture}
~=~\rgtorgbof{\begin{tikzpicture}
	\begin{pgfonlayer}{nodelayer}
		\node [style=none] (0) at (0, 1.5) {};
		\node [style=go] (1) at (0, 0.5) {$2$};
		\node [style=go] (2) at (0, -0.5) {$2$};
		\node [style=none] (3) at (0, -1.5) {};
	\end{pgfonlayer}
	\begin{pgfonlayer}{edgelayer}
		\draw (0.center) to (1);
		\draw (2) to (3.center);
	\end{pgfonlayer}
\end{tikzpicture}}
\end{split}
\end{equation}
\end{proof}

\section{Euler Decomposition}

Let $E$ be the relation on morphisms of $\rg$ defined by the following:
\begin{equation}
\hd
~~\stackrel{E}{\equiv}~~
  \begin{tikzpicture}
	\begin{pgfonlayer}{nodelayer}
		\node [style=none] (0) at (0, 2) {};
		\node [style=go] (1) at (0, 1) {$1$};
		\node [style=ro] (2) at (0, -0) {$1$};
		\node [style=go] (3) at (0, -1) {$1$};
		\node [style=none] (4) at (0, -2) {};
	\end{pgfonlayer}
	\begin{pgfonlayer}{edgelayer}
		\draw (0.center) to (1);
		\draw (2) to (3);
		\draw (3) to (4.center);
		\draw (1) to (2);
	\end{pgfonlayer}
\end{tikzpicture}
\end{equation}

This is known as an Euler decomposition of the Hadamard gate. Duncan and Perdrix~\cite{duncan2009graph} have shown that this equality is not provable in $\rg$.

From this relation, we can produce the quotient category $\rgp:=\rg/E$ with quotient functor $\mathcal{Q}:\rg\to\rgp$. Duncan and Perdrix~\cite{duncan2009graph} showed that Van den Nest's theorem~\cite{van2004graphical} is equivalent to the addition of $E$. This means that Van den Nest's theorem is not provable in $\rg$ but is provable in $\rgp$.

\begin{defn}[chromatic diagram category]
  We will use the term chromatic diagram category---usually denoted by $\cat D$---to mean any category of $\rg$, $\rgb$ or $\rgp$
\end{defn}

\begin{lem}
  For every pair $f\stackrel{E}{\equiv} g$, we have $\rgtorgb{f}=\rgtorgb{g}$ and $\rgtostab{f}=\rgtostab{g}$.
  \label{lem:rgp}
\end{lem}

\begin{proof}
  This is proved by a routine check on all the rules of $\rg$.
\end{proof}

With the previous lemma, we can now prove the following proposition

\begin{prop}
We can lift  $\mathcal{T}$ and $\rgtostab{\cdot}$ uniquely to the functors $\mathcal{\hat T}:\rgp\to\rgb$ and $\rgptostab{\cdot}$ such that the following diagram commutes:
  \begin{equation}
      \xymatrix{
  \rg
\ar@{>>}[rr]^{\mathcal{Q}}
\ar[dr]_{\mathcal{T}}
\ar@/_1.5pc/[dddr]_{\rgtostab{\cdot}}
&&\rgp
\ar[dl]^{\mathcal{\hat T}}
\ar@/^1.5pc/[dddl]^{\rgptostab{\cdot}}\\
&\rgb
\ar[dd]^{\rgbtostab{\cdot}}\\\\
&\stab
}
  \end{equation}
\end{prop}

\begin{proof}
  The lifting of $\mathcal{T}$ and $\rgtostab{\cdot}$ is immediate from the fact that $\rgp$ is a quotient category and Lemma~\ref{lem:rgp}~\cite{mac1998categories}.

  The lower right 2-cell is proved by using the fact that $\mathcal{Q}$ is an epi.
\end{proof}

We can define a functor $\mathcal{S}:\rgb\to\rgp$ to act the following way on generators:
\[\def\r#1{\scalebox{0.7}{$#1$}}
\begin{array}{@{}c@{~\quad}c@{~\quad}c@{~\quad}c@{~\quad}c@{}}
   \rgbtorgpof{\counit{gn}}=~\counit{go}
&  \rgbtorgpof{\rot{gn}}=~\rot{go}
&
  \rgbtorgpof{\unit{gn}}=~\unit{go}
&
  \rgbtorgpof{\comul{gn}}=\comul{go}
&
  \rgbtorgpof{\mul{gn}}=\mul{go}
\\
  \rgbtorgpof{\counit{rn}}=~\counit[$1$]{ro}
&
  \rgbtorgpof{\rot{rn}}=~\rot{ro}
&
   \rgbtorgpof{\unit{rn}}=~\unit[$3$]{ro}
&
  \rgbtorgpof{\comul{rn}}=\comul[$3$]{ro}
&
  \rgbtorgpof{\mul{rn}}=\mul[$1$]{ro}
\\\\[-2ex]
  \rgbtorgpof{\counit{bn}}=~\counit{ro}
&
 \rgbtorgpof{\rot{bn}}=~
  \begin{tikzpicture}
	\begin{pgfonlayer}{nodelayer}
		\node [style=none] (0) at (0, 1.8) {};
		\node [style=go] (1) at (0, 1) {$3$};
		\node [style=ro] (2) at (0, -0) {$\theta$};
		\node [style=go] (3) at (0, -1) {$1$};
		\node [style=none] (4) at (0, -1.8) {};
	\end{pgfonlayer}
	\begin{pgfonlayer}{edgelayer}
		\draw (0.center) to (1);
		\draw (2) to (3);
		\draw (3) to (4.center);
		\draw (1) to (2);
	\end{pgfonlayer}
  \end{tikzpicture}
&
  \rgbtorgpof{\unit{bn}}=~\unit{ro}
&
  \rgbtorgpof{\comul{bn}}=\hspace{-2mm}
\r{\begin{tikzpicture}
	\begin{pgfonlayer}{nodelayer}
		\node [style=none] (0) at (0, 1.8) {};
		\node [style=go] (1) at (0, 0.8) {$3$};
		\node [style=ro] (2) at (0, -0.2) {};
		\node [style=go] (3) at (-0.8, -0.8) {$1$};
		\node [style=go] (4) at (0.8, -0.8) {$1$};
		\node [style=none] (5) at (-0.8, -1.8) {};
		\node [style=none] (6) at (0.8, -1.8) {};
	\end{pgfonlayer}
	\begin{pgfonlayer}{edgelayer}
		\draw (0.center) to (1);
		\draw (4) to (6.center);
		\draw (2) to (3);
		\draw (2) to (4);
		\draw (3) to (5.center);
		\draw (1) to (2);
	\end{pgfonlayer}
\end{tikzpicture}}
&
    \rgbtorgpof{\mul{bn}}=\hspace{-2mm}
\r{\begin{tikzpicture}
	\begin{pgfonlayer}{nodelayer}
		\node [style=none] (0) at (-0.8, 1.8) {};
		\node [style=none] (1) at (0.8, 1.8) {};
		\node [style=go] (2) at (-0.8, 0.8) {$3$};
		\node [style=go] (3) at (0.8, 0.8) {$3$};
		\node [style=ro] (4) at (0, 0.2) {};
		\node [style=go] (5) at (0, -0.8) {$1$};
		\node [style=none] (6) at (0, -1.8) {};
	\end{pgfonlayer}
	\begin{pgfonlayer}{edgelayer}
		\draw (0.center) to (2);
		\draw (1.center) to (3);
		\draw (4) to (5);
		\draw (3) to (4);
		\draw (5) to (6.center);
		\draw (2) to (4);
	\end{pgfonlayer}
\end{tikzpicture}}
\end{array}
\]
\begin{thm}
  $\rgbtorgp$ is the inverse functor of $\hat \rgtorgb$. In other words, $\rgb$ and $\rgp$ are isomorphic categories.
\end{thm}

\begin{proof}
  This requires showing that for each generator $\varphi$ of $\rgb$, one can prove $\hat \rgtorgb(\rgbtorgp{\varphi}) = \varphi$ from the rules of $\rgb$ and that for each generator $\psi$ of $\rgp$, one can prove $\rgbtorgpof{\hat\rgtorgb{\psi}}=\psi$ from the rules of $\rgp$.
\end{proof}

In a sense, the extra equations that appear in $\rgb$ and are not in $\rg$ are the equations about the rotations of the octahedral subgroup of the rotation group of the Bloch sphere. This isomorphism of categories tells us that these rotation equations are necessary and sufficient to prove Van den Nest's theorem.

\section{Extensions}
In the description of the dichromatic calculus given in~\cite{coecke2008interacting,coecke2011interacting,duncan2009graph,duncan2010rewriting}, phases were permitted to be arbitrary real numbers representing arbitrary angles---not just integer multiples of $\frac{\pi}{2}$. We offer a method of extending the present work to make that possible.

In an $n$-chromatic diagram, the group of phases for the combined $n$ colours can be seen as a quotient of $C_4^{*n}$, the free product of $n$ copies of $C_4$. $C_4^{*n}$ thus maps into $\cat D$ in a natural way. If $g: C_4^{*n}\to G$ is a group homomorphism, then we can construct the following pushout:
\begin{equation}
  \xymatrix{
C_4^{*n}
\ar[r]^{g}
\ar[d]_I
\po&
G
\ar[d]\\
\cat{D}
\ar[r]&
\cat{D}(G)
}
\end{equation}

Note that although the homomorphism $g$ is not present in the notation of $\cat D(G)$, the pushout does depend on it. We omit it so as to avoid clutter.

In the case where $\cat D = \rgb$, we can internalize some of the relations in $\rgb$ by using the quotient map we had defined earlier in the proof of Proposition ~\ref{prop:octa}: $q: C_4*C_4*C_4 \to O$. The pushout square then becomes:
\begin{equation}
  \xymatrix{
C_4*C_4*C_4
\ar[r]^{q}
\ar[d]_I
\po&
O
\ar[d]\\
\rgb
\ar[r]_{\id}&
\rgb
}
\end{equation}

We suppose that $H$ is an abelian group and $h:C_4\to H$ is a group homomorphism. Of particular interest to us will be the case where $h$ is monic, in which case, $H$ will be seen as an extension of the phase group. Then we can lift $h$ to a group homomorphism $h^{*n}$ by the unique arrow which makes the following diagram commute:
\begin{equation}
\xymatrix{
&&&C_4^{*n}\ar@{-->}[ddd]^{h^{*n}}\\
C_4\ar[urrr]^{i_0}
\ar[d]_h
&
\cdots&
C_4
\ar[ur]^{i_n}
\ar[d]_h\\
H\ar[drrr]_{j_0}
&
\cdots&
H
\ar[dr]_{j_n}\\
&&&
H^{*n}
}
\end{equation}

$h^{*n}$ is guaranteed to exist and be unique by the universal property of the free product in $\cat{Grp}$.

If $D$ is a chromatic diagram category with $n$ colours, then there is a functor $I:C_4^{*n}\to D$---where $C_4^{*n}$ is seen as a category with one object---which is defined by the property that $I\circ i_k$ maps $C_4$ to the phase group of the $k$th colour. From this data, we can construct the following pushout.
\begin{equation}
  \xymatrix{
C_4^{*n}
\ar[r]^{h^{*n}}
\ar[d]_I
\po&
H^{*n}
\ar[d]\\
\cat{D}
\ar[r]&
\cat{D}(H^{*n})
}
\end{equation}

In the particular case that $H=U(1)$---the group of rotations of the circle---with $h:k\mapsto \frac{\pi}{2}k$, we use the shorthand $\dcirc:=\cat{D}(U(1))$. The diagrams of $\rg_{\bigcirc}$ were the ones described in ~\cite{coecke2011interacting}.

In $\rgb$, we end up with the following cube:
\begin{equation}
\xymatrix{
  C_4^{*3}
\ar[rr]^q
\ar[dr]
\ar[dd]_{h*h*h}
&&
O
\ar[dr]
\ar[dd]
\\
&
\rgb
\ar[rr]_{\id}
\ar[dd]
&&
\rgb
\ar[dd]\\
U(1)^{*3}
\ar[rr]
\ar[dr]
&&
S
\ar[dr]\\
&\rgb_\bigcirc
\ar[rr]_{\id}
&&
\rgb_\bigcirc
}
\end{equation}

where all the faces are pushouts and $S$ is a group extension of $SO(3)$.

There is an obvious way to additionally define a functor $\tostab{\dcirc}{\cdot}:\dcirc\to\fdhilb_Q$ such that the following diagram commutes.
\begin{equation}
\xymatrix{
\cat D
\ar[r]
\ar[d]_{\tostab{\cat D}{\cdot}}
&
\dcirc
\ar[d]^{\tostab{\dcirc}{\cdot}}
\\
\cat{Stab}
\ar[r]
&
\fdhilb_Q
}
\end{equation}

\section{Future Work}
As mentioned earlier, it is known that $\rg$ is not complete for $\stab$---that is to say, $\rgtostab{\cdot}$ is not a faithful functor. It is not known, however if $\rgbtostab{\cdot}$ is faithful or not. In the future, we would hope to either prove that it is (which would be a great result), or show that it's not by providing a counterexample.

A lot of the work we have done was duplicated---once for $\rg$ and again for $\rgb$. It would be a good idea to abstract away the common features of $\rg$ and $\rgb$ to yield a theory of ``coloured graphical calculi''. We could then see if this can be applied to other candidates such as the GHZ/W calculus~\cite{coecke2010compositional}.

Spekkens has developed a toy theory of qubits~\cite{PhysRevA.75.032110}, which was later formalized categorically in ~\cite{coecke2011toy} as the category $\spek$. $\spek$ exposes some features of quantum mechanics, but not all.  For example, it does not ``have'' non-locality~\cite{coecke2011phase}. Interpreting the diagrams of $\rgb$ into $\spek$ would hopefully reveal more of the differences between $\spek$ and $\stab$.

The concept of environment structures was defined in~\cite{coecke2010environment}, to formalize some of the features of classicality, measurements and complementarity. However, these ideas were defined in the context of $\dag$-compact categories. It would be interesting to see how this definition can be adapted to $\rgb$, which is only a \dsmc{}.

On the more practical side, the trichromatic calculus can hopefully be used to derive protocols. Some protocols---such as quantum secret sharing---explicitly use three complementary observables. Hopefully, we will be able to use Quantomatic to help with this task as well.

\bibliography{rgb}
\bibliographystyle{eptcs}


\end{document}